\documentclass[11pt, letterpaper, reqno]{amsart}
\usepackage[utf8]{inputenc} 	
\usepackage{microtype} 			
\usepackage{geometry} 			
\usepackage{fancyhdr} 			
\usepackage{amsmath}			
\usepackage{amsthm}		 		
\usepackage{amssymb}	 		
\usepackage{esint}	 			
\usepackage{bm}					
\usepackage{mathrsfs}			
\usepackage{thmtools}
\usepackage{xcolor}				
\usepackage{tikz}				
\usepackage[all]{xy}			
\usepackage{graphicx}			
\usepackage{wrapfig}
\usepackage{enumitem} 			
\usepackage{booktabs}
\usepackage{array}
\usepackage{lmodern}			
\usepackage[T1]{fontenc}		
\usepackage[colorlinks=true, citecolor=cyan, urlcolor=magenta]{hyperref}
\usepackage{cleveref}
\usepackage[backend=biber, style=alphabetic, backref=true, firstinits=true, isbn=false]{biblatex}
\makeatletter
\renewbibmacro{in:}{}
\DeclareFieldFormat{pages}{#1}
\renewcommand*{\bibnamedash}{%
	\leavevmode\raise +0.6ex\hbox to 5.5ex{\hrulefill}.\space\space}

\InitializeBibliographyStyle{\global\undef\bbx@lasthash}

\newbibmacro*{bbx:savehash}{\savefield{fullhash}{\bbx@lasthash}}

\renewbibmacro*{author}{%
	\ifboolexpr{
		test \ifuseauthor
		and
		not test {\ifnameundef{author}}
	}
	{%
		\iffieldequals{fullhash}{\bbx@lasthash}
		{\bibnamedash\addcomma\space}
		{\printnames{author}}%
		\usebibmacro{bbx:savehash}%
		\iffieldundef{authortype}
		{}
		{%
			\setunit{\addcomma\space}%
			\usebibmacro{authorstrg}%
		}%
	}
	{\global\undef\bbx@lasthash}%
}
\makeatother
\newtheorem{propositionx}{Proposition}[section]
\newenvironment{proposition}
{\pushQED{\qed}\propositionx}
{\popQED\endpropositionx}
\newenvironment{propositionp}
{\pushQED{\qed}\propositionx}
{\popQED\endpropositionx}
\newtheorem{theoremx}[propositionx]{Theorem}
\newenvironment{theorem}
{\pushQED{\qed}\theoremx}
{\popQED\endtheoremx}

\newtheorem{lemmax}[propositionx]{Lemma}
\newenvironment{lemma}
{\pushQED{\qed}\lemmax}
{\popQED\endlemmax}

\theoremstyle{remark}
\newtheorem{remark}{Remark}
\newcommand{\bbB}{\mathbb{B}}
\newcommand{\bbC}{\mathbb{C}}

\newcommand{\bbH}{\mathbb{H}}

\newcommand{\bbN}{\mathbb{N}}

\newcommand{\bbR}{\mathbb{R}}
\newcommand{\bbS}{\mathbb{S}}
\newcommand{\bbT}{\mathbb{T}}

\newcommand{\bbZ}{\mathbb{Z}}

\newcommand{\calF}{\mathcal{F}}

\newcommand{\calJ}{\mathcal{J}}

\newcommand{\calL}{\mathcal{L}}

\newcommand{\calR}{\mathcal{R}}

\newcommand{\calT}{\mathcal{T}}

\newcommand{\scrD}{\mathscr{D}}

\newcommand{\scrS}{\mathscr{S}}

\newcommand{\frakR}{\mathfrak{R}}

\newcommand{\dd}{\,\mathrm{d}}

\title{Long time behavior of the half-wave trace and Weyl remainders}
\author{Ethan Sussman}
\date{March 4th, 2022 (Last Updated), September 20th, 2021 (V1)}
\email{ethanws@mit.edu}
\address{Department of Mathematics, Massachusetts Institute of Technology, Massachusetts 02139-4307, USA}

\begin{document}

\begin{abstract}
	Given a compact Riemannian manifold $(M,g)$, Chazarain, H\"ormander, Duistermaat, and Guillemin study the half-wave trace $\operatorname{HWT}_{M,g}(\tau) \in \scrS'(\bbR_\tau)$. From the asymptotics of the half-wave trace as $\tau\to 0$, H\"ormander deduces the now standard  remainder $\smash{O(\sigma^{d-1}) = O(\lambda^{d/2-1/2})}$ in Weyl's law, where $d=\dim M$. Given a dynamical assumption implying additional local regularity, Duistermaat and Guillemin improve this to $o(\sigma^{d-1})$. By examining the Tauberian step in the argument, we show how a quantitative version 
	\[
	N(\sigma) = Z(\sigma) + O(\sigma^{d-1}\calR(\sigma)^{-1/2})
	\]
	of the Duistermaat-Guillemin result follows under slightly stronger hypotheses, these implying that the $(d-1)$-fold regularized half-wave trace 
	\[
	\langle D_\tau \rangle^{1-d} \operatorname{HWT}_{M,g}(\tau)
	\]
	is in $\smash{L^{1,1}_\mathrm{loc}(\bbR\backslash \{0\})}$. Here $Z(\sigma)\in \bbR[\sigma]$ is a polynomial and $\calR(\sigma):\bbR^+\to \bbR^+$ is an $(M,g)$-dependent nondecreasing function with $\lim_{\sigma\to\infty} \calR(\sigma)= \infty$, specified in terms of the growth rate of $\langle D_\tau \rangle^{1-d} \tau^{-1}\operatorname{HWT}_{M,g}(\tau)$ as measured in $L^{1,1}$. Per Duistermaat-Guillemin, this hypothesis is implied by geometric conditions that hold ``generically'' for $d\geq 3$. Thus, we clarify the relation between the error term in Weyl's law and the long time behavior of the half-wave trace. 
\end{abstract}

\maketitle

\tableofcontents

\section{Introduction}

In this paper we apply a modification of an argument -- originally used by Newman \cite{Newman} in order to prove the prime number theorem -- to the proof of Weyl-type laws, e.g.\ Weyl's law \cite{Weyl1911}\cite{Levitan}\cite{Avakumovic1956}
\begin{align}
	N(\lambda) &= (2\pi)^{-d} \operatorname{Vol}_g(M) \operatorname{Vol}(\bbB^d) \lambda^{d/2} + O(\lambda^{d/2-1/2}) \label{eq:Weyl_0}\\
	&= (2\pi)^{-d} \operatorname{Vol}_g(M) \operatorname{Vol}(\bbB^d) \Sigma^{d} + O(\Sigma^{d-1}),
	\label{eq:Weyl_1}
\end{align}
where the function $N: [0,\infty)\to \bbN$, $N(\lambda) = \#\{n : \lambda_n \leq \lambda\}$ is the counting function of the eigenvalues $\lambda_0,\lambda_1,\lambda_2,\ldots$ of the Laplace-Beltrami operator on a compact Riemannian manifold $(M,g)$, without boundary.   
(Going forward, we will work with $\Sigma = \lambda^{1/2}$ instead of $\lambda$, following the notational conventions of \cite{DG}.) 
One goal is to estimate, under appropriate geometric or analytic assumptions, the remainder $N(\lambda) - (2\pi)^{-d} \operatorname{Vol}_g(M) \operatorname{Vol}(\bbB^d) \lambda^{d/2}$. 
For recent work, along with a more extensive bibliography, see \cite{Canzani2019EigenfunctionCV}\cite{JG}\cite{JG2}. The standard estimate, $O(\lambda^{d/2-1/2})$, as written in \cref{eq:Weyl_1}, is sharp when $(M,g)$ is a Zoll manifold -- e.g. the $d$-sphere, $\bbS^d$ -- so any improvement must be contingent on some geometric or analytic assumptions.

One method for estimating the remainder involves using a Tauberian theorem for the Laplace transform in $\lambda$. The standard version of this argument yields only an $o(\lambda^{d/2})$ remainder in \cref{eq:Weyl_0}.
Newman, in his proof \cite{Newman}\cite{Korevaar82}\cite{Zagier}\cite{KTextbook} of Ingham's Tauberian theorem, assumed what amounted to continuity of the Fourier-Laplace transform of the function $\alpha$ under investigation all the way up to and including the imaginary axis (sans the origin). See the expositions by Zagier \cite{Zagier} and Korevaar \cite{Korevaar82}\cite{KTextbook} for this perspective, which differs slightly from Newman's. 

In a series of papers,  H\"ormander \cite{Hormander68} (inspired by \cite{Levitan}\cite{Avakumovic1956}), Chazarain \cite{Chazarain1974}, and Duistermaat \& Guillemin \cite{DG} established, under geometric hypotheses, the precise regularity of the half-wave trace $-i\tau \calF N_{1/2}(\tau)$ along the whole real axis. Here, $N_{1/2}(\Sigma) = N(\Sigma^2)$. 
In particular, under an appropriate genericity assumption, the worst singularity of the half-wave trace is at the origin. When fed into Newman's Tauberian argument, this yields a generic improvement of \cref{eq:Weyl_1} in which the remainder is smaller than that in \cref{eq:Weyl_1} by an $o(1)$ factor, as originally shown using a different argument in \cite[Theorem 3.5]{DG}. When supplemented with a bound on the growth of the half-wave trace as measured in an appropriate Sobolev space (a Sobolev space in which the half-wave trace of a Zoll manifold cannot lie, even locally), we can replace Duistermaat-Guillemin's $o(1)$ with a function specified in terms of the growth rate of the half-wave trace. Demonstrating this fact is the goal of the present paper. 

As exemplified by \cite{Canzani2019EigenfunctionCV}\cite{JG}\cite{JG2} and the references there, explicit quantitative improvements on Duistermaat-Guillemin's result follow from generic dynamical assumptions. The state of the art uses sophisticated microlocal and semiclassical tools. Our goal here is merely to clarify, using nothing more than a simple Tauberian argument, how improved estimates are encoded in the structure of the half-wave trace. 

We record here the main Tauberian theorem. It is a version of Ingham's Tauberian theorem with remainder, a more general version of which is stated in \Cref{thm:Ingham}. 
\begin{theorem}
	\label{thm:main_1}
	Suppose that $N: \bbR^{\geq 0} \to (0,\infty)$ is a piecewise-continuous, nondecreasing function, $Z(\sigma)\in \bbC[\sigma]$ is a polynomial in $\sigma$ of degree $d=\deg Z$ and $E\in C_{\mathrm{c}}^\infty((0,\infty))$. 
	Let 
	\begin{equation} 
		\alpha(\sigma) = \begin{cases}
			\langle \sigma\rangle^{1-\deg Z} (N(\sigma)-Z(\sigma)+E(\sigma)) & (\sigma\geq 0) \\
			0 & (\sigma< 0),
		\end{cases}
	\end{equation}
	and suppose that the singular support of $\calF N$ and hence $a= \calF\alpha$ is discrete. Suppose further that 
	\begin{enumerate}
		\item $\alpha \in  L^\infty(\bbR)$, 
		\label{it:taub}
		\item $a(\tau)$ is in $\tau L^{1,\ell}_{\mathrm{loc}}(\bbR_\tau)$ for some integer $\ell\geq 1$.
		\label{it:ab}
	\end{enumerate}
	Then, for $\calR_0(\Sigma) = \max\{1,\sup\{T>0: T \lVert t^{-1} a(t) \rVert_{L^{1,\ell}[-T,+T]} \leq \Sigma^\ell \}\}$, 
	\begin{equation} 
		N(\Sigma) = Z(\Sigma)  + O( \Sigma^{d-1} \calR_0(\Sigma)^{-1/2})
	\end{equation}  
	as $\Sigma\to\infty$.  
\end{theorem}
And the main application to Weyl's law: 
\begin{theorem}
	\label{thm:main}
	Suppose that $(M,g)$ is a compact Riemannian manifold such that 
	\begin{equation} 
	\tau^{-1} \langle D_\tau \rangle^{1-d} \tau^{-1} \operatorname{HWT}_{M,g}(\tau) \in L_{\mathrm{loc}}^{1,\ell}(\bbR\backslash \{0\})
	\end{equation} 
	for some $\ell\in \bbN^+$. 
	Then, for some polynomial $Z_0(\sigma) \in \bbC[\sigma]$ of degree at most $d-2$, 
	\begin{equation}
		N(\lambda) = (2\pi)^{-d} \operatorname{Vol}_g(M) \operatorname{Vol}(\bbB^d) \Sigma^d + Z_0(\Sigma)  + O(\Sigma^{d-1} \calR(\Sigma)^{-1/2} )
	\end{equation}
	as $\Sigma\to\infty$, where $\calR(\Sigma) = \max\{1,\sup\{T>0: T \lVert \tau^{-1} \langle D_\tau \rangle^{1-d} (\tau^{-1} \operatorname{HWT}_{M,g}(\tau)) \rVert_{L^{1,\ell}[1,T]_\tau} \leq \Sigma^\ell \}\}$. 
\end{theorem}
Note that $\calR(\Sigma)$ satisfies $\lim_{\Sigma\to\infty} \calR(\Sigma)=\infty$, so $O(\Sigma^{d-1} \calR(\Sigma)^{-1/2} ) = o(\Sigma^{d-1})$. 

Combined with \Cref{prop:apriori}, \Cref{thm:main} yields the following (somewhat loosely stated) theorem:
\begin{theorem}
	\label{thm:final}
	Suppose that $(M,g)$ is a nonpathological $d$-dimensional compact Riemannian manifold (in the sense of satisfying the hypotheses of \cite[Theorem 4.5]{DG} for all $T$) whose nontrivial geodesic loops are at most $\delta$-fold degenerate, $\delta \in \{0,\ldots,2d-1\}$. 
	
	Consider $\ell \in \{1,\ldots, \lceil d-\delta/2-1/2\rceil -1\}$. 
	
	Then, $\langle D_\tau \rangle^{1-d} (\tau^{-1} \operatorname{HWT}_{M,g}(\tau)) \in L^{1,\ell}_{\mathrm{loc}}(\bbR\backslash \{0\})$, and for some polynomial $Z_0(\sigma) \in \bbC[\sigma]$ of degree at most $d-2$, 
	\begin{equation}
	N (\lambda) = (2\pi)^{-d} \operatorname{Vol}_g(M) \operatorname{Vol}(\bbB^d) \Sigma^d + Z_0(\Sigma)  + O(\Sigma^{d-1} \calR(\Sigma)^{-1/2} )
	\end{equation}
	as $\Sigma\to\infty$, where $\calR(\Sigma) = \sup\{1,T>0: T \lVert \tau^{-1} \langle D_\tau \rangle^{1-d} (\tau^{-1} \operatorname{HWT}_{M,g}(\tau)) \rVert_{L^{1,\ell}[1,T]_\tau} \leq \Sigma^\ell \}$. 
\end{theorem}
\begin{remark}
	Of course, $Z_0$ can be computed in terms of the Laurent series of $\operatorname{HWT}_{M,g}(\tau)$ at $\tau = 0$, though we will not do so explicitly below.  
\end{remark}
\begin{remark}
	Roughly speaking, $\nu=2d-1-\delta$ is the codimension of the subset of the sphere bundle consisting of generators of periodic bicharacteristics. In order to apply the previous theorem, we need $\nu  \geq 3$. 
	Since the Duistermaat-Guillemin improvement \cite[Theorem 3.5]{DG} of Weyl's law applies (loosely speaking) when the set of generators of periodic bicharacteristics has positive codimension $\nu\geq 1$, \Cref{thm:main} applies in slightly less generality. However, we expect that a version of \Cref{thm:main} does apply to all Riemannian manifolds to which the Duistermaat-Guillemin theorem applies, the restriction above being more born out of a desire to avoid fractional regularity $L^1$-based function spaces than any apparent need. 
	
	See \S\ref{sec:final} for a more precise discussion.
\end{remark}
Debruyne \cite[\S6]{Debruyne} has recently proven results with a similar flavor to \Cref{thm:main_1}.

We point out the following two special cases of \Cref{thm:main}: 
\begin{itemize}
	\item a polynomial bound on the half-wave trace (as measured in an appropriate Sobolev space), e.g.\ $\langle D_\tau \rangle^{1-d} \tau^{-1} \operatorname{HWT}_{M,g}(\tau) \in L_{\mathrm{loc}}^{1,\ell}(\bbR\backslash \{0\}) \cap \langle \tau \rangle^k L^{1,\ell}[1,\infty)_\tau $ for $\ell\in \bbN$ and $k\geq 1$, yields a polynomial improvement of Weyl's law: 
	\begin{equation}
		N(\lambda) = (2\pi)^{-d} \operatorname{Vol}_g(M) \operatorname{Vol}(\bbB^d) \Sigma^d + O(\Sigma^{d - 1 - \epsilon})
		\label{eq:Weyl_2}
	\end{equation}
	for $\epsilon= \ell/2k>0$. Indeed, in this case, $\calR(\Sigma) = \Omega( \Sigma^{\ell/k})$. 
	
	For instance, the half-wave trace on the $d$-torus $\bbT^d=\bbR^d/\Lambda$, $\Lambda$ a full rank lattice, grows polynomially in the relevant sense (as can be proven via the method of images) (although actually we should regularize the half-wave trace by an extra $\langle D \rangle^{-\epsilon}$ for arbitrarily small $\epsilon>0$ before carrying out the argument below in order to get a sharp \emph{integral} amount of Sobolev regularity).
	\item On the other hand, if $\langle D_\tau \rangle^{1-d} \tau^{-1} \operatorname{HWT}_{M,g}(\tau) \in L_{\mathrm{loc}}^{1,\ell}(\bbR\backslash \{0\}) \cap e^{c\tau^\beta} L^{1,\ell}[1,\infty)_\tau $ for some $c>0$ and $\beta>0$, then we have 
	\begin{equation}
		N(\lambda) = (2\pi)^{-d} \operatorname{Vol}_g(M) \operatorname{Vol}(\bbB^d) \Sigma^d + O(\Sigma^{d - 1} / \log^{1/2\beta} \Sigma ).
	\end{equation}
	Indeed, $\calR(\Sigma) =  \Omega( \log^{1/\beta} \Sigma )$ in this case.
\end{itemize}
Note that we do not in the present paper prove any bounds on the half-wave trace.

We now outline the proof of \Cref{thm:main_1}. While the proof is not long (being modeled on Newman's short proof of Ingham's Tauberian theorem), we believe the central idea is worth presenting once at a low level of detail (and with a somewhat philosophical emphasis). There exist many variants of Ingham's Tauberian theorem and correspondingly many proofs. Newman's -- especially as presented by  \cite{Korevaar82}\cite{Zagier} -- makes manifest use of the behavior of the Fourier-Laplace transform along the imaginary axis, which is in a precise sense (\Cref{lem:bdy}) the Fourier transform. While this hypothesis is unnecessarily strong  for proving Ingham's Tauberian theorem, it is precisely what is needed here to make use of the rich structure of the half-wave trace revealed by \cite{Chazarain1974}\cite{Hormander68}\cite{DG}.  

Applied to the Laplace transform, the idea behind Newman's proof of Ingham's theorem is as follows: given a nondecreasing (typically piecewise-constant) function $N(\sigma) \in \smash{\langle \sigma \rangle^{K} L^\infty(\bbR^{\geq 0}_\sigma)}$, for some ${K} \in \bbR$, and given an ansatz $Z_0(\sigma) \in \langle \sigma \rangle^{K} L^\infty(\bbR^{\geq 0}_\sigma)$ (typically a polynomial, possibly corrected by an element of $C_{\mathrm{c}}^\infty((0,\infty))$) for $N$, attempt to determine the asymptotics of the improper integral
\begin{equation}
	A_\Sigma = \int_0^\Sigma \frac{1}{\langle \sigma \rangle^J} (N(\sigma) - Z_0(\sigma)) \dd \sigma 
	\label{eq:improper}
\end{equation}
as $\Sigma\to\infty$, 
for appropriate $J\in \bbR$. (Here $\langle \sigma \rangle = (1+\sigma^2)^{1/2}$ is the ``Japanese bracket.'') If $Z_0(\sigma)$ is a good polynomial ansatz (and the Fourier transform of the integrand is sufficiently regular around zero), then the integrand $\langle \sigma \rangle^{-J} (N(\sigma)-Z(\sigma))$ should be oscillating, and the integral of an oscillating function is typically smaller than we would expect via computing an $L^1$-norm. 
Consequently, in order to understand the $\Sigma\to\infty$ asymptotics of $A_\Sigma$, we should expect to have to understand the amplitude of these oscillations -- that is, understand the Fourier transform away from zero. 
In typical applications, ${K}\in \bbN$, with the general case introducing only notational complications, and we take $Z_0$ to be a polynomial of degree ${K}$ outside of some neighborhood of the origin. When estimating the remainder in Weyl's law, $Z_0(\sigma)$ can be taken to be
\begin{equation} 
	Z_0(\sigma)=\Theta(\sigma) Z(\sigma)+E(\sigma)
\end{equation}  
for $Z \in \bbR[\sigma]$ of degree $d$, $E \in C_{\mathrm{c}}^\infty(\bbR^+)$ (or perhaps Schwartz). 
Here $\Theta:\bbR\to \{0,1\}$ is the Heaviside step function, $\Theta(\sigma)=(1+\operatorname{sign} \sigma)/2$.
In this case, for all $J\geq {K}$ sufficiently large, $A_\Sigma$ will converge as $\Sigma\to\infty$, at some polynomial rate which can be estimated. 
(And $E$ can be chosen so as to make $\lim_{\Sigma \to \infty} A_\Sigma = 0$, if so desired.)
Suppose that $J$ is in fact sufficiently large, leaving aside for the moment the question of what precisely this means (which will, of course, depend on the specifics).
Since $N$ is nondecreasing, if for some $\sigma_0>0$ the quantity $N(\sigma_0)$ differs too much from its leading order asymptotic, as captured by the ansatz $Z(\sigma_0)$, then the integrand above will be comparably large over an interval of polynomial length; this will result in a relatively large contribution to the tail of the improper integral (\ref{eq:improper}). If this event were to occur for infinitely many, arbitrarily large $\sigma_0 $, then this sets a  bound for the rate of convergence of $A_\Sigma$ as $\Sigma\to\infty$. If this bound contradicts the known rate of convergence, then we can conclude that the difference between $N(\sigma)$ and $Z(\sigma)$ is smaller than supposed, at least for sufficiently large $\sigma$. See \S\ref{sec:second}, where this argument (a basic ``mean-to-max'' argument) is made precise. 

The crux of the argument, then, is estimating the quantity $A_\Sigma$ defined by \cref{eq:improper} to a sufficiently high degree of accuracy. 
In Newman's Tauberian approach, this is done by comparing $A_\Sigma$ with 
\begin{equation}
	A_\Sigma(t) = \int_0^\Sigma   \frac{e^{-\sigma t}}{\langle \sigma\rangle ^J} (N(\sigma) - Z_0(\sigma)) \dd \sigma, \quad  A(t) = \int_0^\infty   \frac{e^{-\sigma t}}{\langle \sigma \rangle^J} (N(\sigma) - Z_0(\sigma)) \dd \sigma, 
\end{equation}
defined initially for $t>0$. 
As $t\to 0^+$ for fixed $\Sigma>0$, $A_\Sigma(t) \to A_\Sigma$ (indeed, $A_\Sigma(t)$ is entire in $t$), and as $\Sigma\to \infty$ for fixed $t>0$, $A_\Sigma(t) \to A(t)$. So, by taking $t\to 0^+$ and $\Sigma\to\infty$ together at some appropriate relative rate, and by estimating the closeness of $A_\Sigma(t),A(t),A_\Sigma$, we can estimate $A_\Sigma$. 
If $Z$ is chosen correctly, then $A(t)$ (in the applications at hand) will actually be continuous all the way down to $t=0$, so it makes sense to first take $t\to 0^+$ and then take $\Sigma\to\infty$. 
Estimating the closeness of $A_\Sigma(0)$ and $A(0)$ is accomplished by means of a complex analytic argument whose input is the  regularity of the Fourier transform $\calF \alpha$ of 
\begin{equation} 
	\alpha(\sigma) = 
	\begin{cases}
		\langle \sigma \rangle^{-J}(N(\sigma) - Z_0(\sigma)) & (\sigma>0) \\
		0 & (\sigma \leq 0)
	\end{cases}
\end{equation} 
in the interval $[-T,+T]$, for each $T>0$. (We will mainly take $J=d-1$, but modifications of the arguments below can deal with other $J$.) Indeed, $A(t)$ is the Laplace transform of $\langle\sigma\rangle^{-J} (N(\sigma) - Z_0(\sigma))$, evaluated at $t$.  Hence, $A(t)$ analytically continues to $t \in \mathbb{H}_{\mathrm{R}}$, $\bbH_{\mathrm{R}}\subset \bbC$ the right half of the complex plane. This is the Fourier-Laplace transform, and the Fourier transform is its distributional boundary value at the imaginary axis $i\bbR_\tau \subset \bbC_t$. See \S\ref{sec:primer}. Trivially, $L^1$-based Sobolev regularity of $\calF \alpha$ implies $L^\infty$-based decay of $\alpha$  -- indeed, it implies even better decay by the Riemann-Lebesgue lemma -- but Newman's argument is stronger, precisely because it works even if $\calF\alpha$ fails to decay, and this is definitely necessary for an application to Weyl's law. In other words, Newman's argument works even if $\calF \alpha$ only has local $L^1$-based Sobolev regularity. If a tempered distribution such as $a$ fails to decay, this might signal the irregularity of the distribution's Fourier transform $\alpha$, but even if the half-wave trace has rapid growth we already know that $\alpha$ is a function. 
This fact allows us to make up for the lack of decay by estimating (when $A \in \tau L^1[-\epsilon,+\epsilon]_\tau$ for some $\epsilon>0$) the closeness of the full Fourier transform $\calF_{\tau \to \Sigma}^{-1}(\tau^{-1} A(i \tau))(\Sigma)$ and the partial Fourier transform 
\begin{equation}
	\frac{1}{2\pi} \int_{-T}^{+T} e^{i \tau \Sigma}   A(i \tau) \frac{\dd \tau}{\tau}
	\label{eq:68g}
\end{equation}
in terms of $L^\infty$-bounds on $\alpha$. The integral (\ref{eq:68g}) corresponds to one component of Newman's contour integral, and the other components are bounded in terms of assumed $L^\infty$-bounds on $\alpha$, while (\ref{eq:68g}) is bounded in terms of the regularity of the half-wave trace. 
This is not completely correct, because some smoothing of the integrand at $\tau=\pm T$ is necessary, but such smoothing is easily handled. 

In the application to Weyl's law, $a(\tau) = \calF \alpha(\tau)$ is essentially the (regularized, weighted) half-wave trace away from the origin,  weighted by $1/\tau$. By Chazarain \cite[\S0]{Chazarain1974}, the singular support of $a$ is the set $\calT \subset \bbR$ of lengths of all geodesic loops (a.k.a.\ ``periodic bicharacteristics'') in $(M,g)$, including loops which consist of multiple laps around a basic loop. Generically, $\calT$ is a discrete set, and the half-wave trace $\operatorname{HWT}_{(M,g)}(\tau)$ will be given by a Laurent series 
\begin{equation}
	\operatorname{HWT}_{(M,g)}(\tau) = \sum_{s\in S} \frac{a_s}{(\tau-T-i0)^s} \bmod C^\infty(\bbR_\tau)\text{ near $T$}, 
	\label{eq:laurent}
\end{equation}
for 
$S(T)\subset \bbZ^+ \cup (1/2) \bbZ^+$ finite and $T$-dependent $\{a_s\}_{s\in S}\subset \bbC$, 
in some neighborhood of each $T \in \calT$. It is therefore to be expected that the size of the remainder in Weyl's law be estimated in terms of 
\begin{enumerate}[label=(\Roman*)]
	\item the size of the elements of $S(T)$ and $\{a_s\}_{s\in S(T)}$ in \cref{eq:laurent}, as $T$ varies,
	\item the density of $\calT \subset \bbR$ (roughly speaking), and 
	\item the growth rate of the smooth remainder in \cref{eq:laurent}.
\end{enumerate}
The case $T=0$ is different from the case $T\neq 0$; the pole of $a$ in the former case yields contributions to $A_\Sigma$ that are one power of $\Sigma$ larger than the poles of the same order in the latter case. (So, due to the weighting by $1/\tau$, the pole at $T=0$ of the half-wave trace yields contributions to $A_\Sigma$ that are two powers larger than the poles of the same order in the latter case.) 

The subsequent sections of this paper are organized as follows: 
\begin{itemize}
	\item \S\ref{sec:primer} is a primer on the Fourier-Laplace transform, 
	\item \S\ref{sec:main} contains the bulk of the Tauberian argument for estimating the asymptotics of $A_\Sigma$, including a flexible and somewhat refined version (\Cref{thm:Ingham}) of Ingham's theorem, with a strong estimate of the remainder,
	\item \S\ref{sec:second} contains the mean-to-max step of the Tauberian argument,  the extraction of weak asymptotics of $N(\Sigma)$ from the asymptotics of $A_\Sigma$ via monotonicity (finishing the proof of \Cref{thm:main_1}), and
	\item \S\ref{sec:final} contains the application to Weyl's law, including the proof of \Cref{thm:main}. 
\end{itemize}  
The harmonic analysis utilized in \S\ref{sec:final} is somewhat overkill, but it does yield a clean statement and proof of the main theorem.

It is worth noting that we do not take into account any cancellations between the various contributions to $N(\Sigma)$ coming from the poles in (\ref{eq:laurent}) for different $T\neq 0$ (though \Cref{thm:Ingham} is stated in sufficient generality to make use of such cancellations if they are known). We simply apply the triangle inequality.  
When $(M,g)$ is a standard flat torus $\bbR^d/\bbZ^d$ of dimension $d\geq 2$ (for which the problem of estimating $N(\Sigma)$ is the Gauss $d$-sphere problem), such cancellations do in fact occur. See \cite{Ivic} for the best known results for the Gauss $d$-sphere problem. See also \cite{Huxley} for an extensive discussion of the $d=2$ case, along with related techniques.

\section{Primer on the Fourier-Laplace transform}
\label{sec:primer} 

Let $\scrS'(\bbR^{\geq 0}) \subset \scrS'(\bbR)$ denote the set of tempered distributions supported on the nonnegative real axis.  That is, $u \in \scrS'(\bbR)$ is in $\scrS'(\bbR^{\geq 0})$ if and only if $u(\chi)=0$ whenever $\operatorname{supp} \chi \subset (-\infty,0)$.  
In this section, we review the \emph{Fourier-Laplace transform}, defined as a map $\calF\calL : \scrS'(\bbR^{\geq 0}) \to A(\bbH_{\mathrm{R}})$, where $A(\bbH_{\mathrm{R}})$ is the set of analytic functions on $\bbH_{\mathrm{R}} = \{z\in \bbC: \Re z>0\}$. (The Fourier-Laplace transform actually makes sense on distributions supported on translates of $\bbR^{\geq 0}$, as is clear from the discussion below, but we have no use for this level of generality, although this may be useful in other applications of Newman's argument.) Detailed presentations of the Laplace transform on tempered distributions can be found in a number of textbooks, e.g. \cite{Dijk2013}, so we will be brief, presenting only the results crucial to the rest of the paper.

The map is defined as follows. For any $\alpha \in \scrS'(\bbR^{\geq 0})$ and $t \geq 0$, the product $\alpha_t(\sigma)=\exp(-t \sigma) \alpha(\sigma) \in \scrS'(\bbR_\sigma^{\geq 0})$ is well-defined, and in fact if $t>0$ we have $\alpha_t \in \cup_{\ell \in \bbZ}  \cap_{k\in \bbZ} \langle \sigma \rangle^k \langle D \rangle^\ell L^2(\bbR_\sigma)$, so the Fourier transform 
\begin{equation}
	\calF \alpha_t(\tau) = \int_{-\infty}^{+\infty} e^{-i\tau \sigma - t \sigma } \alpha(\sigma) \dd \sigma \in \bigcup_{\ell \in \bbZ} \bigcap_{k\in \bbZ} \langle \tau \rangle^\ell  \langle D \rangle^k  L^2(\bbR_\tau) 
\end{equation}
is well-defined, the integral above being formal if $\alpha$ is not a subexponential function (but typically a Lebesgue integral in the cases of interest). By the Sobolev embedding theorem, we deduce that $\calF \alpha_t(\tau)$ is actually a smooth function of $\tau$. Consequently, we can define a function $\calF \calL \alpha :\bbH_{\mathrm{R}} \to \bbC$ by $\calF \calL \alpha(t+i\tau) = \calF \alpha_t(\tau)$, $t>0$ and $\tau \in \bbR$. 

\begin{proposition}
	\label{prop:FLA}
	Given $\alpha \in \scrS'(\bbR^{\geq 0})$, 
	$\calF \calL \alpha$ is analytic. 
\end{proposition}
For the following proof, we recall the following facts about tempered distributions $u \in \scrS'(\bbR^2)$ of the form $u(x,y)=u_1(x)u_2(y)$ for $u_1,u_2 \in \scrS'(\bbR)$: 
\begin{itemize}
	\item $\partial_x u(x,y) = (\partial_x u_1(x)) u_2(y)$, 
	\item $\partial_y y(x,y) = u_1(x) (\partial_y u_2(y))$,
\end{itemize}
the partial Fourier transform in either of $x,y$ is well-defined and has the expected mapping properties, and the distributional derivatives of any smooth distribution agree with the classical derivatives (and likewise for a.e.\ derivatives, etc.). 
\begin{proof}
	For each $t_0>0$, let $\alpha^{(t_0)}(t,\sigma)=1_{t\geq t_0} e^{-t\sigma} \alpha(\sigma)$, which is a well-defined element of $\scrS'(\bbR_t\times \bbR_\sigma)$.
	Let 
	\begin{equation}
		1_{t\geq t_0} \calF \calL \alpha(t+i\tau) = 
		\begin{cases}
			0 & (t<t_0) \\
			\calF\calL \alpha(t+i\tau) & (t \geq t_0), 
		\end{cases}
	\end{equation}
	which we can interpret as an element of $\scrS'(\bbR_t\times \bbR_\tau)$. Clearly, this is the partial Fourier transform of $\alpha^{(t_0)}(t,\sigma)$ along $\bbR_\sigma$. 
	Fix $\chi \in C_{\mathrm{c}}(\bbC)$ supported in $2t_0+\bbH_{\mathrm{R}}$. 
	We see that 
	\begin{equation}
		\chi(z)\frac{\partial^2 }{\partial t^2}  1_{t\geq t_0} \calF \calL \alpha(t+i\tau)=   +\chi(z)  1_{t\geq t_0} \calF \calL (\sigma^2 \alpha(\sigma))(t+i\tau),
	\end{equation}
	\begin{equation}
		\chi(z)\frac{\partial^2 }{\partial \tau^2}  1_{t\geq t_0} \calF \calL \alpha(t+i\tau) = - \chi(z) 1_{t\geq t_0} \calF \calL (\sigma^2 \alpha(\sigma))(t+i\tau),
	\end{equation}
	which means that $\calF \calL \alpha(t+i\tau)$ is harmonic -- and therefore smooth by elliptic regularity -- in the interior of the support of $\chi$. By the arbitrariness of $\chi$ and $t_0$, we deduce that $\calF \calL\alpha$ is actually smooth in $\bbH_{\mathrm{R}}$. 
	
	Clearly, $\calF\calL\alpha$ satisfies the Cauchy-Riemann equations in the distributional sense (by analogous reasoning), from which it follows that it satisfies the Cauchy-Riemann equations in the usual classical sense. We conclude that $\calF \calL$ is analytic, as claimed. 
\end{proof}
So, $\alpha \mapsto \calF \calL \alpha$ defines a map $\calF \calL : \scrS'(\bbR^{\geq 0}) \to A(\bbH_{\mathrm{R}})$, and this is the Fourier-Laplace transform.

Note that not all elements of $\scrS'(\bbR^{\geq 0})$ that are locally integrable admit a Laplace transform in the sense of a Lebesgue integral: 
\begin{equation}
	\begin{cases}
		e^{\sigma} \cos(e^{\sigma}) & (\sigma\geq 0) \\
		0 & (\sigma <0) 
	\end{cases}
	\label{eq:cnr}
\end{equation}
defines an element of $\scrS'(\bbR_\sigma^{\geq 0}) \cap L^1_{\mathrm{loc}}(\bbR_\sigma)$, but the Laplace transform of (\ref{eq:cnr}) does not make sense as a Lebesgue integral. Nevertheless, we have:
\begin{propositionp}
	Suppose that $\alpha \in \cup_{k\in \bbN} \langle \sigma \rangle^k L^1(\bbR_\sigma)$, so that $e^{-t\sigma} \alpha(\sigma) \in L^1(\bbR_\sigma)$ for each $t>0$. Then,
	\begin{equation}
		\calF \calL \alpha(t+i\tau) = \int_0^\infty e^{-i \tau \sigma - t\sigma} \alpha(\sigma) \dd \sigma, 
	\end{equation}
	where the right-hand side is now a well-defined Lebesgue integral. 
\end{propositionp}
This is just a special case of the statement that Schwartz's and Lebesgue's Fourier transforms agree on $L^1(\bbR)$. 

We now discuss distributional boundary values. 

\begin{proposition}
	\label{lem:bdy}
	For any $\alpha \in \scrS'(\bbR^{\geq 0})$, the indexed set $\{\calF \calL \alpha (t+i\tau)\}_{t>0} \subset \scrS'(\bbR_\tau)$ satisfies $\calF \calL \alpha(t+i \tau)\to \calF \alpha(\tau)$ in $\scrS'(\bbR_\tau)$, i.e.\ $\calF \alpha$ is the ``distributional boundary value'' of $\calF \calL \alpha$ on the imaginary axis. 
\end{proposition}
\begin{proof}
	Since $\calF \calL \alpha (t+i\tau) = \calF  \alpha_t(\tau)$, it suffices (since $\calF : \scrS'(\bbR)\to \scrS'(\bbR)$ is continuous) to show that $\alpha_t(\sigma)\to \alpha(\sigma)$ in $\scrS'(\bbR_\tau)$ as $t\to 0^+$. This latter claim follows from the proposition that if $\varphi \in \scrS(\bbR)$ vanishes identically for sufficiently negative input, then $e^{-t \sigma } \varphi(\sigma)$ converges to $\varphi(\sigma)$ in $\scrS(\bbR_\sigma)$. This is the statement that, for all $\varphi$ as above, $ \sigma^\ell \partial^k_\sigma (e^{-t \sigma} \varphi(\sigma) - \varphi(\sigma)) \to 0$ as $t\to 0^+$ in $L^\infty(\bbR_\sigma)$ for each $\ell, k\in \bbN$. By Leibniz's rule, this follows from the statement that   
	\begin{equation}
		t^k e^{-t \sigma} \varphi(\sigma) \to 0 \quad \text{ and }  \quad (e^{-t \sigma} \varphi(\sigma) - \varphi(\sigma)  ) \to 0 
		\label{eq:19g}
	\end{equation}
	in $L^\infty(\bbR_\sigma)$ 
	for all $k\geq 1$ and all $\varphi \in \scrS(\bbR)$ vanishing identically for sufficiently negative input. The first is clear, since $\lVert t^k e^{-t \sigma} \varphi(\sigma) \rVert_{L^\infty(\bbR_\sigma)} \leq t^k  \lVert \varphi \rVert_{L^\infty} \sup_{\sigma\in \operatorname{supp} \varphi} e^{-t \sigma}$. 
	
	The second statement in \cref{eq:19g} holds because $\varphi(\sigma) \in \langle \sigma \rangle^{-1} L^\infty(\bbR_\sigma)$.
\end{proof}

The following proposition (while likely not strictly necessary and stated without proof in \cite{DG}) can be convenient in simplifying the technical details of Newman's argument. (This simplification is removed in e.g. \cite[Chapter III, \S14]{KTextbook}\cite{debruyne2019complex}, and we expect that the modified argument carries over, \emph{mutatis mutandis}.)
This result is a special case of a result in the theory of boundary problems for elliptic differential operators.
\begin{propositionp}
	\label{prop:smoothness}
	Suppose that $\calF \alpha \in \scrS'(\bbR)$ is smooth in some neighborhood of $T\in \bbR$. Then there exists some relatively open neighborhood $U\subset \smash{\overline{\bbH}_{\mathrm{R}}}$ of $iT$ in which $\calF \calL \alpha$ extends smoothly to the boundary. 
\end{propositionp}

\section{Tauberian argument: Newman's computation}
\label{sec:main}

The cornerstone of Newman's Tauberian argument is the following complex-analytic computation, expressing the  integral 
\begin{equation}
	A: [0,\infty)_\Sigma \to \bbR , \quad A:\Sigma\mapsto A_\Sigma, \quad A_\Sigma = \int_0^\Sigma \alpha(\sigma) \dd \sigma 
\end{equation}
of a tempered, locally integrable (possibly complex-valued) function supported on the nonnegative real axis -- denoted $\alpha \in \scrS'(\bbR^{\geq 0})\cap L^1_{\mathrm{loc}}(\bbR^{\geq 0})$ throughout this section -- in terms of 
\begin{itemize}
	\item a contribution from the principal part of the Laurent series of $a(\tau) = \calF \alpha(\tau)$ at $\tau = 0$ (whose well-definedness will be one of our hypotheses),
	\item a separately treated contribution from the constant term of the Laurent series,
	\item a contribution from the singularities $T\in \calT\backslash \{0\}$ of $a(\tau)$ in $[-T,+T]$ away from $\tau = 0$,
	\item  a typically small contribution from the remainder of $a(\tau)$ on $[-T,+T]$, and
	\item  a controllable contribution from $a(\tau)$ outside of $[-T,+T]$,
\end{itemize}
where $T\geq 1$ is a parameter of the argument.
We will not be explicit about the contribution to $A_\Sigma$ from the off-origin poles of $a$, which can be treated as an error to be controlled (and in applications should typically be the same size of the error from $a(\tau)$ outside of $[-T,+T]$). We handle the off-origin singularities of $a(\tau)$ separately from the smooth background because while it is most convenient to measure the latter in an $L^1$-based Sobolev space, in order to be sharp it is necessary to measure the former as a pseudomeasure, i.e.\ in $\calF \langle \sigma \rangle^K L^\infty(\bbR_\sigma)$ for some $K\geq 0$.

Based on our treatment of each contribution, we decompose
\begin{equation}
	a(\tau) = a_{00} \calF \chi_0 +a_0 (\tau) +b (\tau) + c(\tau), 
	\label{eq:decomp}
\end{equation} 
where $a_{00}\in \bbC$, $\chi_0 \in C_{\mathrm{c}}^\infty(\bbR^+)$ satisfies $\int_{0}^\infty \chi_0(\sigma) \dd \sigma = 1$, and 
\begin{itemize}
	\item $\smash{a_0(\tau) = \sum_{j \in \calJ} a_{0}^{(j)} (\tau-i0)^{-j}}$ for some finite subset $\calJ\subset \bbR^+$ of positive real numbers and some indexed collection $\smash{\{a_0^{(j)}\}_{j\in \calJ}\subset \bbC}$,
	\item  $c(\tau) \in \tau L^1_{\mathrm{loc}}(-\epsilon,+\epsilon) \cap \scrS'(\bbR_\tau)$ is a tempered distribution (with discrete singular support) of the form $c=\calF \gamma$ for 
	\begin{equation}
		\gamma(\sigma) \in \scrS'(\bbR^{\geq 0}) \cap \cup_{K\geq 0}\langle \sigma \rangle^K L^\infty(\bbR_\sigma^{\geq 0}), 
		\label{eq:gamma_inc}
	\end{equation}  and
	\item $b=a-a_0-c$ is the remainder, which we will not analyze (but will in applications typically be zero). 
\end{itemize}
We will further split $c(\tau) =p(\tau) +g(\tau)$, where $g(\tau)$ is regular (so $p$ absorbs all of the off-origin poles). The function $g$ is the portion of $c$ that will be measured in an $L^1$-based Sobolev space, while $p$ will be analyzed in $\calF \langle \sigma \rangle^K L^\infty(\bbR_\sigma)$. 
Given the setup above, $b$ can be the entirety of $a$. So no matter what $\alpha$ is, we can find $a_{00},a_0,b,c$ as above such that (\ref{eq:decomp}) holds, but consequently the results below are only interesting when the contribution from $b$ can be analyzed somewhat explicitly --- e.g.\ if $b=0$ (which, as mentioned above, will be the case in typical applications). 

The main result of this section, presented in \Cref{thm:Ingham}, is a result of the following form: given $\alpha,a,a_{00}, a_0,b,c,g,p$ as above, and given appropriate control on $g,p$,
\begin{equation}
	\int_0^\Sigma \alpha(\sigma) \dd \sigma = \sum_{j\in \calJ} \frac{e^{\pi i j/2}}{j!} a_{0}^{(j)} \Sigma^j+\int_0^\Sigma \beta(\sigma) \dd \sigma + a_{00}  + \frac{1}{\calR(\Sigma)}
\end{equation}
for sufficiently large $\Sigma$, 
where $\beta=\calF^{-1} b$ is the inverse Fourier transform of $b$ and $\calR(\Sigma)>0$ is to-be-specified. This result is to be thought of as an estimation of the left-hand side by the right-hand side sans the $1/\calR(\sigma)$ term.
The $1/\calR(\Sigma)$ term is to be thought of as the estimate's error, bounded in terms of the global Sobolev regularity of $g$, the size of $\calF p$, and the assumed growth rate of $\gamma$ (that is, the $K$ witnessing \cref{eq:gamma_inc}). Hence, if e.g.\ $b,\beta=0$ and the function $1/\calR(\Sigma)$ is sufficiently slowly growing (or, better yet, decaying), then we understand the $\Sigma\to\infty$ asymptotics of $A_\Sigma$, at least to what is generically leading order. Our estimate is closely related to -- but not obviously equivalent to -- a result of Debruyne \cite[\S6]{Debruyne}, but our the argument is different (following Newman).

The key computation is:
\begin{proposition}
	\label{prop:setup}
	Suppose that $\alpha  \in  \scrS'(\bbR^{\geq 0})\cap L^1_{\mathrm{loc}}(\bbR^{\geq 0})$. Let $a= \calF \alpha$.
	Suppose we are given some distributions $b,c \in \scrD'(\bbR)$ of the form 
	\begin{enumerate}
		\item $b = \calF \beta$ for $\beta \in  \scrS'(\bbR^{\geq 0})\cap L^1_{\mathrm{loc}}(\bbR^{\geq 0})$ and 
		\item $c = \calF \gamma$ for $\gamma \in  \scrS'(\bbR^{\geq 0})\cap L^1_{\mathrm{loc}}(\bbR^{\geq 0})$
	\end{enumerate} 
	such that $a=b+c$ and $c$ is continuous in some neighborhood of the origin. 
	Then, if $T>0$ is such that $c$ is smooth in some neighborhood of the two points $-T,+T\in \bbR$, we can write,  for each $\Sigma>0$
	\begin{equation}
		\int_0^\Sigma \alpha(\sigma) \dd \sigma = \int_0^\Sigma \beta(\sigma) \dd \sigma + c(0)+ I_1(T,\Sigma,M)+I_2(T,\Sigma,M)+I_3(T,\Sigma,M),
		\label{eq:r10}
	\end{equation}
	for arbitrary $M\in \bbN^+$, where $I_1$ is a well-defined Riemann integral given by \cref{eq:I11},  
	\begin{align*}
	I_1(T,\Sigma,M) &= \frac{1}{2\pi i} \int_{\Gamma_+(T,0)}  (C_\Sigma(t) - C(t)) e^{t\Sigma}\Big( 1 + \frac{t^2}{T^2} \Big)^M \frac{\dd t}{t}, 
	\intertext{and $I_2,I_3$ are well-defined distributional pairings given by (\ref{eq:I2}, \ref{eq:I3}),} 
	I_2(T,\Sigma,M) &= -\frac{1}{2\pi i } \int_{-T}^{+T} C_\Sigma(i\tau) e^{i \tau \Sigma}\Big( 1 - \frac{\tau^2}{T^2} \Big)^M \frac{\dd \tau}{ \tau+  i0} ,\\ 
	I_3(T,\Sigma,M) &= +\frac{1}{2\pi i} \int_{-T}^{+T} c(\tau) e^{i \tau \Sigma}\Big( 1 - \frac{\tau^2}{T^2} \Big)^M \frac{\dd \tau}{ \tau+  i0},
	\end{align*}
	where $\Gamma_+(T,0)$ is the positively oriented contour whose path is the semicircle arc of radius $T$. 
\end{proposition}
\begin{proof}
	For each $T>0$ and $\delta \in [0,T)$, let $\Gamma=\Gamma(T,\delta)$ denote the boundary of the region $
	\{z\in \bbC : |z| \leq T, \Re z \geq \delta \}$. 
	We consider $\Gamma$ as a contour, oriented counterclockwise. 
	Set $\Gamma_0(T,\delta) = \{z\in \bbC: |z| \leq T, \Re z = \delta\}$ and $\Gamma_+(T,\delta)=\{z\in \bbC: |z| = T, \Re z > \delta\}$,  so that
	\begin{equation} 
		\Gamma(T,\delta) = \Gamma_+(T,\delta) \cup \Gamma_0(T,\delta).
	\end{equation} 
	We consider $\Gamma_0$ and $\Gamma_+$ as contours, with orientations consistent with that of $\Gamma$. (That is, $\Gamma_0$ is oriented top-to-bottom, and $\Gamma_+$ is oriented counterclockwise.)
	
	Let $A,B,C : \bbH_{\mathrm{R}} \to \bbC$ denote the (well-defined) Fourier-Laplace transforms 
	\begin{equation}
		A(t) = \int_0^\infty \alpha(\sigma) e^{-\sigma t} \dd \sigma, \quad
		B(t) = \int_0^\infty \beta(\sigma) e^{-\sigma t} \dd \sigma, \quad 
		C(t) = \int_0^\infty \gamma(\sigma) e^{-\sigma t} \dd \sigma.
	\end{equation}
	(These integrals may have to be interpreted formally if $\alpha,\beta,\gamma$ are not in $\cup_{K\geq 0} \langle \sigma \rangle^K L^1(\bbR_\sigma)$, but $A = \calF\calL \alpha,B=\calF\calL \beta,C=\calF\calL\gamma$ are always well-defined functions, as shown in the previous section.)
	For each $\Sigma\geq 0$, let $A_\Sigma,B_\Sigma,C_\Sigma:\bbC\to \bbC$ be the \emph{entire} functions defined by the Lebesgue integrals
	\begin{equation}
		A_\Sigma(t) = \int_0^\Sigma \alpha(\sigma)  e^{-\sigma t} \dd \sigma, \quad 
		B_\Sigma(t) = \int_0^\Sigma \beta(\sigma)  e^{-\sigma t} \dd \sigma, \quad C_\Sigma(t) =\int_0^\Sigma \gamma(\sigma)  e^{-\sigma t} \dd \sigma. 
	\end{equation}
	(We are assuming that $\alpha,\beta,\gamma$ are locally integrable. Entirety is proven as in \Cref{prop:FLA}.)
	Note that, by \Cref{lem:bdy}, $a,b,c$, are the distributional boundary values of $A,B,C$ (respectively) at the imaginary axis. 
	
	By the Cauchy integral theorem: for each $t_0 \in (\delta,T)$ and $M\in \bbN^+$, 
	\begin{equation}
		(C_\Sigma(t_0) - C(t_0)) e^{t_0 \Sigma} \Big( 1 + \frac{t_0^2}{T^2} \Big)^M  = \frac{1}{2\pi i} \int_\Gamma  (C_\Sigma(t) - C(t)) e^{t\Sigma}\Big( 1 + \frac{t^2}{T^2} \Big)^M \frac{\dd t}{t-t_0}. 
		\label{eq:o7}
	\end{equation}  
	We split up the integral in (\ref{eq:o7}) into three pieces:
	\begin{equation}
		\frac{1}{2\pi i} \int_\Gamma  (C_\Sigma(t) - C(t)) e^{t\Sigma}\Big( 1 + \frac{t^2}{T^2} \Big)^M \frac{\dd t}{t-t_0} = I_1 + I_2 + I_3,
		\label{eq:z09}
	\end{equation}
	where $I_1=I_1(\delta,t_0,T,\Sigma,M)$, $I_2=I_2(\delta,t_0,T,\Sigma,M)$, and $I_3 = I_3(\delta,t_0,T,\Sigma,M)$ are functions of $\delta,t_0,T,\Sigma,M$ given by
	\begin{align}
		I_1 &=+\frac{1}{2\pi i} \int_{\Gamma_+(T,\delta)}  (C_\Sigma(t) - C(t)) e^{t\Sigma}\Big( 1 + \frac{t^2}{T^2} \Big)^M \frac{\dd t}{t-t_0}, \\ 
		I_2 &= +\frac{1}{2\pi i}  \int_{\Gamma_0(T,\delta)}  C_\Sigma(t) e^{t\Sigma}\Big( 1 + \frac{t^2}{T^2} \Big)^M \frac{\dd t}{t-t_0}, \\	
		I_3 &= -\frac{1}{2\pi i}  \int_{\Gamma_0(T,\delta)}  \;\:C(t) e^{t\Sigma}\Big( 1 + \frac{t^2}{T^2} \Big)^M \frac{\dd t}{t-t_0}.
	\end{align}
	We now consider the limit $\delta \downarrow 0$, for fixed $t_0,T,\sigma,M$. Clearly, $\lim_{\delta\to 0^+} I_2(\delta,t_0,T,\Sigma,M)$ exists and is equal to 
	\begin{align}
		I_2(t_0,T,\Sigma,M) &= \frac{1}{2\pi i}  \int_{\Gamma_0(T,0)}  C_\Sigma(t)  e^{t\Sigma}\Big( 1 + \frac{t^2}{T^2} \Big)^M \frac{\dd t}{t-t_0}  \\ &= -\frac{1}{2\pi i}  \int_{-T}^{+T} C_\Sigma(i\tau) e^{i \tau \Sigma}\Big( 1 - \frac{\tau^2}{T^2} \Big)^M \frac{\dd \tau}{ \tau+  it_0}. 
		\label{eq:l39}
	\end{align}
	Similarly, since $\pm T \notin \operatorname{singsupp} C$ -- and therefore $C$ is continuous in some neighborhood of $\pm i T$ in $\overline{\bbH}_{\mathrm{R}}$, by \Cref{prop:smoothness} -- $\lim_{\delta\to 0^+} I_1(\delta,t_0,T,\Sigma,M)$ exists and is equal to 
	\begin{equation}
		I_1(t_0,T,\Sigma,M) = \frac{1}{2\pi i} \int_{\Gamma_+(T,0)}(C_\Sigma(t) - C(t)) e^{t\Sigma}\Big( 1 + \frac{t^2}{T^2} \Big)^M \frac{\dd t}{t-t_0}, 
		\label{eq:l310}
	\end{equation}
	which is well-defined as a Riemann integral. The $\delta\downarrow 0$ limit of $I_3$ requires a bit of care.
	Since $\pm T\notin \operatorname{singsupp}c$ (and singular supports are closed sets, by definition), we can find some $T_0 \in (0,T)$ such that  $[-T,+T] \cap \operatorname{singsupp} c \subseteq [-T_0,+T_0]$. Choose some 
	$\chi \in C_{\mathrm{c}}^\infty(\bbR)$ which is supported in the interior of $[-T,+T]$ and such that $\chi$ is identically equal to one on the slightly smaller interval $[-T+(T-T_0)/2,+T-(T-T_0)/2]$. Then we can write 
	\begin{align}
		I_3(\delta,t_0,T,\Sigma,M) &= -\frac{1}{2\pi i}  \int_{\Gamma_0(T,\delta)} \chi(\Im t) C(t)  e^{t\Sigma}\Big( 1 + \frac{t^2}{T^2} \Big)^M \frac{\dd t}{t-t_0}  \\ &-\frac{1}{2\pi i}  \int_{\Gamma_0(T,\delta)}  (1-\chi(\Im t)) C(t)  e^{t\Sigma}\Big( 1 + \frac{t^2}{T^2} \Big)^M \frac{\dd t}{t-t_0}.
		\label{eq:al9}
	\end{align}
	As $\delta\to 0^+$, the first term on the right-hand side converges to 
	\begin{equation}
		\frac{1}{2\pi i}\int_{-T}^{+T} \chi( \tau) c(\tau)  e^{i \tau\Sigma}\Big( 1 - \frac{\tau^2}{T^2} \Big)^M \frac{\dd \tau}{\tau+it_0} = \frac{1}{2\pi i} \Big\langle c(\tau), \chi(\tau) e^{i \tau\Sigma} \Big( 1 - \frac{\tau^2}{T^2} \Big)^M \frac{1}{\tau+ it_0} \Big\rangle_{\scrS'(\bbR_\tau)} 
		\label{eq:s6}
	\end{equation}
	by the definition of distibutional convergence, and the second term, (\ref{eq:al9}), converges to the well-defined Riemann integral 
	\begin{equation}
		\frac{1}{2\pi i} \int_{-T}^{+T} (1-\chi(\tau)) c(\tau)  e^{i \tau \Sigma}\Big( 1 - \frac{\tau^2}{T^2} \Big)^M \frac{\dd \tau }{\tau + it_0}
		\label{eq:s8} 
	\end{equation}
	since $C(t)$ is continuous on some neighborhood in $\overline{\bbH}_{\mathrm{R}}$ of the compact set $[-T,-T+(T-T_0)/2] \cup [+T-(T-T_0)/2,+T]$. Combining (\ref{eq:s6}) and (\ref{eq:s8}): $\lim_{\delta\to 0^+} I_3(\delta,t_0,T,\Sigma,M)$ exists (although this we could deduce from the limits already considered) and is equal to the well-defined distributional pairing
	\begin{equation}
		I_3(t_0,T,\Sigma,M) =  \frac{1}{2\pi i} \int_{-T}^{+T} c(\tau) e^{i \tau \Sigma} \Big( 1 - \frac{\tau^2}{T^2} \Big)^M \frac{\dd \tau}{\tau+ i t_0}.
	\end{equation}
	Combining what we have shown so far, 
	\begin{equation}
		(C_\Sigma(t_0) - C(t_0) ) e^{t_0 \Sigma} \Big( 1 + \frac{t_0^2}{T^2} \Big)^M  = I_1(t_0,T,\Sigma,M)+I_2(t_0,T,\Sigma,M)+I_3(t_0,T,\Sigma,M), 
		\label{eq:p0}
	\end{equation}
	where  $I_1,I_2,I_3$ are given by \cref{eq:l310}, \cref{eq:l39}, \cref{eq:al9}, respectively. 
	
	We now consider the limit $t_0 \to 0^+$, for fixed $T,\Sigma,M$. The limit $\lim_{t_0\to 0^+} I_1(t_0,T,\Sigma,M)$ clearly exists and is equal to
	\begin{equation}
		I_1(T,\Sigma,M) = \frac{1}{2\pi i} \int_{\Gamma_+(T,0)}  (C_\Sigma(t) - C(t)) e^{t\Sigma}\Big( 1 + \frac{t^2}{T^2} \Big)^M \frac{\dd t}{t}.  
		\label{eq:I11}
	\end{equation} 
	Since $C_\Sigma|_{i\bbR}$ (by entirety) and $c$ (by assumption) are both continuous in some neighborhood of the origin, $\lim_{t_0 \to 0^+} I_2(t_0,T,\Sigma,M)$ and $\lim_{t_0 \to 0^+} I_3(t_0,T,\Sigma,M)$ exist and are equal to the well-defined distributional pairings 
	\begin{align}
		I_2(T,\Sigma,M) &= -\frac{1}{2\pi i } \int_{-T}^{+T} C_\Sigma(i\tau) e^{i \tau \Sigma}\Big( 1 - \frac{\tau^2}{T^2} \Big)^M \frac{\dd \tau}{ \tau+  i0} \label{eq:I2},\\ 
		I_3(T,\Sigma,M) &= +\frac{1}{2\pi i} \int_{-T}^{+T} c(\tau) e^{i \tau \Sigma}\Big( 1 - \frac{\tau^2}{T^2} \Big)^M \frac{\dd \tau}{ \tau+  i0}. \label{eq:I3}
	\end{align}
	It follows from (\ref{eq:p0}) and the computations of the limits above that 
	\begin{align}
		C_\Sigma(0) = \lim_{t_0\to 0^+} C_\Sigma(t_0) e^{t_0 \Sigma} \Big( 1 + \frac{t_0^2}{T^2} \Big)^M  &=  \lim_{t_0\to 0^+} \Big[C(t_0) e^{t_0 \Sigma} \Big( 1 + \frac{t_0^2}{T^2} \Big)^M + \sum_{j=1}^3 I_j(t_0,T,\Sigma,M) \Big]  \\
		&= c(0)+ I_1(T,\Sigma,M) + I_2(T,\Sigma,M)+I_3(T,\Sigma,M).
	\end{align}
	Since $C_\Sigma(0) = A_\Sigma(0)- B_\Sigma(0)$, we have $A_\Sigma(0) = B_\Sigma(0) + c(0)+ I_1(T,\Sigma,M) + I_3(T,\Sigma,M)+I_2(T,\Sigma,M)$, 
	which is precisely the desired result, \cref{eq:r10}. 
\end{proof}

\subsection{Estimates of the errors}

We now bound the contribution to the integral $A_\Sigma$ coming from $I_1(T,\Sigma,M),I_2(T,\Sigma,M)$, given an \emph{a priori} polynomial bound 
\begin{equation}
	\gamma \in \cup_{K \geq 0} \langle \sigma \rangle^K L^\infty(\bbR_\sigma), 
	\label{eq:ass}
\end{equation}
so that $\gamma \in \langle \sigma \rangle^K L^\infty(\bbR_\sigma)$ for some real number $K \geq 0$. This weak bound (involving no detailed asymptotics of $\gamma$, only an upper bound of the absolute value) functions as the ``Tauberian hypothesis'' of the initial part of Newman's Tauberian argument (the part discussed in the present section). In the application to Weyl's law, this hypothesis is verified by Weyl's law with the standard remainder. (So the standard form of Weyl's law will be used with some assumptions on the half-wave trace in order to deduce a strengthened form of Weyl's law.) Allowing $K>0$ is typically not necessary (since we could always work with more heavily weighted $\alpha$), but it is convenient. The most important case is $K=0$. In this subsection only ``momentum-space'' -- i.e.\ $\bbR_\sigma$-side -- assumptions come into play. 

The $M=1,K=0$ case of the next proposition can be found at the end of \cite{Zagier}:

\begin{proposition}
	\label{prop:setup_base}
	Consider the setup of \Cref{prop:setup}, and suppose that (\ref{eq:ass}) holds, with $\gamma \in \langle \sigma \rangle^K L^\infty(\bbR_\sigma)$, $K\geq 0$. Then, if $M\geq K+1$, $I_1(T,\Sigma,M)$ and $I_2(T,\Sigma,M)$ obey the bounds
	\begin{equation}
		|I_1(T,\Sigma,M)|,|I_2(T,\Sigma,M)| \leq  R_{K,M} \Sigma^{K}\digamma /T
		\label{eq:both_bds}
	\end{equation}
	for $T\geq 1$ and sufficiently large $\Sigma>\Sigma_{K}>0$,
	where  $\digamma = \lVert \gamma \rVert_{\langle \sigma\rangle^K L^\infty(\bbR_\sigma)}$ and $R_{K,M}>0$ is some constant depending only on $K$ and $M$.
\end{proposition}
\begin{proof}
	We handle $I_1,I_2$ in that order. In the following, $R$ (with or without subscripts) will denote an arbitrary constant (which will change line to line). One computation that appears in both bounds  in \Cref{eq:both_bds} is that, on the circle $T\bbS^1=\{t\in \bbC:|t|=T\}$ of radius $T>0$, $|1+t^2/T^2| = |2\Re t/T|$.
	\begin{enumerate}
		\item  On $\Gamma_+(T) \backslash\{\pm i T\}=\Gamma_+(T,0)\backslash\{\pm i T\}$, 
		\begin{equation}
			|C_\Sigma(t)-  C(t)| = \Big| \int_{\Sigma}^\infty \gamma(\sigma) e^{-\sigma t} \dd \sigma  \Big| \leq  \digamma \int_\Sigma^\infty \langle \sigma \rangle^K e^{-\sigma \Re t}  \dd \sigma. 
		\end{equation} 
		For sufficiently large $\Sigma_K>0$, for all $\Sigma> \Sigma_K$ we can bound $\langle \sigma \rangle^K$ for $\sigma\geq \Sigma$ above by $2 \sigma^K$, so
		\begin{equation}
			|C_\Sigma(t)-  C(t)| \leq  2\digamma  \int_{\Sigma}^\infty \sigma^K e^{-\sigma \Re t} \dd \sigma = 2\digamma  \frac{1}{(\Re t)^{K+1}} \int_{\Sigma \Re t}^\infty \sigma^K e^{-\sigma } \dd \sigma
			\label{eq:e93}
		\end{equation}
		for such $\Sigma$. 
		If $K$ is a nonnegative integer, we have the explicit computation
		\begin{equation}
			\int_{\Sigma \Re t}^\infty \sigma^K e^{-\sigma } \dd \sigma = 
			\sum_{k=0}^K k! \binom{K}{k} (\Sigma \Re t)^{K-k} e^{- \Sigma \Re t} \leq R_K e^{- \Sigma \Re t} \langle \Sigma \Re t\rangle^K.
		\end{equation}
		If $K$ is not an integer, then, observing that $\sup_{\sigma\geq \Sigma \Re t} (1+\sigma^{\lceil K \rceil})^{-1} \sigma^K = O(\langle \Sigma \Re t \rangle^{K-\lceil K \rceil}) $,  we have 
		\begin{align}
			\int_{\Sigma \Re t}^\infty \sigma^K e^{-\sigma } \dd \sigma &= \int_{\Sigma \Re t}^\infty \Big( \frac{\sigma^K}{1+\sigma^{ \lceil K \rceil }} \Big) (1+\sigma^{\lceil K \rceil}) e^{-\sigma } \dd \sigma \\ 
			&\leq \sup_{\sigma\geq \Sigma \Re t}\{ (1+\sigma^{\lceil K \rceil})^{-1} \sigma^K\}  \int_{\Sigma \Re t}^\infty  (1+\sigma^{\lceil K \rceil}) e^{-\sigma } \dd \sigma \\
			&\leq R_K e^{- \Sigma \Re t} \langle \Sigma \Re t \rangle^{K}.
		\end{align}
		So, regardless of whether or not $K$ is an integer, $\int_{\Sigma \Re t}^\infty \sigma^K e^{-\sigma} \dd \sigma \leq R_K e^{-\Sigma \Re t} \langle \Sigma \Re t \rangle^K$. 
		
		Plugging this into \cref{eq:e93}, $|C_\Sigma(t) - C(t)| \leq 2 \digamma R_K (\Re t)^{-K-1} e^{-\Sigma \Re t} \langle \Sigma \Re t \rangle^K$.
		
		Now, 
		\begin{equation}
			\Big| e^{t \Sigma} \Big( 1 + \frac{t^2}{T^2} \Big)^M \frac{1}{t} \Big| = e^{ \Sigma\Re t} \cdot \Big(\frac{2 \Re t}{T} \Big)^M \cdot \frac{1}{T}
		\end{equation}
		on $\Gamma_+(T)$. So, we have the following bound on the integrand of $I_1(T,\Sigma,M)$:
		\begin{equation}
			\Big| (C_\Sigma(t) - C(\sigma))e^{t \Sigma} \Big( 1 + \frac{t^2}{T^2} \Big)^M \frac{1}{t} \Big| \leq 2^{M+1} \digamma R_K (\Re t)^{M-1-K} \langle \Sigma \Re t \rangle^K  \frac{1}{T^{M+1}}.
			\label{eq:3g6}
		\end{equation}
		Under the assumption that $M\geq 1+K$, $(\Re t)^{M-1-K} \leq \langle \Re t \rangle^{M-1-K}$ on $\Gamma_+(T)$, so the right-hand side above is bounded above by $R_{K,M} \digamma  \langle \Sigma \rangle^K \langle \Re t \rangle^{M-1} / T^{M+1}$. Redefining $R_{K,M}$, this is bounded above by $R_{K,M} \digamma  \langle \Sigma \rangle^K / T^2$. For $\Sigma$ sufficiently large, this is bounded above by $R_{K,M} \digamma \Sigma^K / T^2$, for some new $R_{K,M}$. 
		
		It follows that $I_1(T,\Sigma,M)$ is bounded above in absolute value by $R_{K,M} \Sigma^{K}\digamma /T$, as claimed, for some new $R_{K,M}$. 
		\item Since $C_\Sigma$ (along with the rest of the integrand of $I_2(\delta,t_0,T,\Sigma,M)$, for $\delta>0$) is analytic to the left of $t_0$, we can deform the contour $\Gamma_0(T,\delta)$ to the (slightly extended) semicircle $\Gamma_-(\delta)=\{t\in \bbC: |t|=T, \Re t\leq \delta\}$ oriented counterclockwise:
		\begin{equation}
			I_2(\delta,t_0,T,\Sigma,M) = \frac{1}{2\pi i} \int_{\Gamma_-(T,\delta)}C_\Sigma(t) e^{t\Sigma} \Big( 1 + \frac{t^2}{T^2} \Big)^M \frac{\dd t}{t-t_0}. 
		\end{equation}
		Obviously, we can take $\delta\to 0^+$, so $I_2(t_0,T,\Sigma,M) = (2\pi i)^{-1} \int_{\Gamma_-(T)}C_\Sigma(t) e^{t\Sigma} (1+t^2/T^2)^M (t-t_0)^{-1}\dd t$, where $\Gamma_-(T)=\Gamma_-(T,0) = \{t\in \bbC : |t|=T, \Re t\leq 0\}$, oriented as before. Since $\Gamma_-(T)$ is bounded away from the origin of the complex plane, the limit of the integral as $t_0\to 0^+$ is well-defined (as a Riemann integral, not just as a distributional pairing) and equal to the contour integral 
		\begin{equation}
			I_2(T,\Sigma,M) =  \frac{1}{2\pi } \int_{\Gamma_-(T)}C_\Sigma(t) e^{t\Sigma} \Big( 1 + \frac{t^2}{T^2} \Big)^M \frac{\dd t}{t}. 
			\label{eq:33f}
		\end{equation}
		On $\Gamma_-(T)$, $t\notin i\bbR$, we can bound 
		\begin{equation}
			|C_\Sigma(t) | = \Big| \int_0^\Sigma \gamma(\sigma) e^{-\sigma t} \dd \sigma \Big| \leq \digamma \int_{0}^\Sigma \langle \sigma \rangle^K e^{+\sigma |\Re t|} \dd \sigma. 
		\end{equation}
		For sufficiently large $\Sigma_K$ (and for some $R_K>0$), for all $\Sigma > \Sigma_K$ we can bound the right-hand side by  $|\Re t|^{-1}\cdot \digamma R_{K} e^{\Sigma |\Re t|} \Sigma^{K} $. And:
		\begin{equation}
			\Big| e^{t \Sigma} \Big( 1 + \frac{t^2}{T^2} \Big)^M \frac{1}{t} \Big| = e^{-\Sigma|\Re t|} \cdot \Big(\frac{2 |\Re t|}{T} \Big)^M \cdot \frac{1}{T}\leq  e^{-\Sigma|\Re t|} \cdot \frac{2^{M} |\Re t|}{T^2}
		\end{equation}
		on $\Gamma_-(T)$.  So, the right-hand side of \cref{eq:33f}, which is an integral of a continuous function along a contour of length $\pi T$, is bounded above in absolute value by $R_{K,M} \Sigma^{K} \digamma/T$, as claimed.
	\end{enumerate}
\end{proof}

We now turn to 
\begin{equation} 
	I_3(T,\Sigma,M) =  \frac{1}{2\pi i} \int_{-T}^{+T} c(\tau) e^{i \tau \Sigma} \Big( 1 - \frac{\tau^2}{T^2} \Big)^M \frac{\dd \tau}{\tau + i0}. 
\end{equation} 
which is evidently where the regularity of $c$ becomes relevant. 

\begin{proposition}
	\label{prop:setup2}
	Consider the setup of \Cref{prop:setup},  $T,-T\notin\operatorname{singsupp} c$. If $Q(\tau)=\tau^{-1}c(\tau)$ is in $L^{1,N}[-T,+T]$ for $N\in \bbN$ and $M\geq N$, $I_3(T,\Sigma,M)$ obeys the bound 
	\begin{equation}
		|I_3(T,\Sigma,M)| \leq R_{N,M} \Sigma^{-N} \lVert Q \rVert_{ L^{1,N}[-T,+T]}, 
		\label{eq:3g9}
	\end{equation}
	for $T\geq 1$, 
	where $R_{N,M}>0$ is a constant depending only on $N$, $M$. 
\end{proposition}
\begin{proof}
	In terms of $Q$, $I_3(T,\Sigma,M)$ is given by 
	\begin{equation}
		I_3(T,\Sigma,M) = \frac{1}{2\pi i} \int_{-T}^{+T} Q(\tau) e^{i \tau \Sigma} \Big(1- \frac{\tau^2}{T^2} \Big)^M \dd \tau.  
	\end{equation}
	One of our hypotheses is that $c$ is smooth in some neighborhoods of $\pm T$, so we can integrate by parts. Fixing $M\geq N$, 
	\begin{equation}
		I_3(T,\Sigma,M) =  \frac{1}{2\pi i} \frac{1}{(i \Sigma)^N} \int_{-T}^{+T} Q(\tau) (\partial_\tau^N e^{i\tau \Sigma} ) \Big(1- \frac{\tau^2}{T^2} \Big)^M \dd \tau = \sum_{n=0}^N J_n, 
		\label{eq:integration_by_parts_formula}
	\end{equation}
	where, for $n=0,\ldots,N$, $J_n=J_n(T,\Sigma,M)$ is given by 
	\begin{equation}
		J_n = (-1)^{N} \frac{1}{2\pi i} \binom{N}{n} \frac{1}{(i\Sigma)^N} \int_{-T}^{+T} e^{i\tau \Sigma}(\partial_\tau^{N-n} Q(\tau))  \partial_\tau^n \Big(1-\frac{\tau^2}{T^2}\Big)^M \dd \tau.
	\end{equation}
	Note the absence of boundary terms when performing the integration by parts --- this follows from our assumption that $M\geq N$. We then have 
	\begin{align}
	\begin{split} 
	\Big|  \int_{-T}^{+T} e^{i\tau \Sigma}(\partial_\tau^{N-n} Q(\tau))  \partial_\tau^n \Big(1-\frac{\tau^2}{T^2}\Big)^M \dd \tau \Big| &\leq \lVert \partial_\tau^{N-n} Q(\tau) \rVert_{L^1[-T,+T]} \Big\lVert  \partial_\tau^n \Big(1-\frac{\tau^2}{T^2}\Big)^M \Big\rVert_{L^{\infty}[-T,+T]} \\
	&= \lVert (\partial_\tau^{N-n} Q(\tau)) \rVert_{L^1[-T,+T]} \lVert  T^{-n}\partial_\tau^n (1-\tau )^M\rVert_{L^{\infty}[-1,+1]} \\
	&= T^{-n}\lVert (\partial_\tau^{N-n} Q(\tau)) \rVert_{L^1[-T,+T]} \lVert  \partial_\tau^n (1-\tau)^M\rVert_{L^{\infty}[-1,+1]} \\
	&\leq  T^{-n} \lVert  Q \rVert_{L^{1,N}[-T,+T]} \lVert  \partial_\tau^n (1-\tau )^M \rVert_{L^{\infty}[-1,+1]}.
	\end{split} 
	\end{align} 
	Consequently, $\smash{|J_n| \leq R_{N,M,n} \Sigma^{-N} T^{-n} \lVert Q^{(N-n)} \rVert_{L^1[-T,+T]}}$ for some constant $R_{N,M,n}>0$. Adding up the contributions from each $n=0,\ldots,N$, we conclude \cref{eq:3g9}.
\end{proof}

\subsection{Contribution from the origin}
\label{subsec:origin}

If $a$ has a pole at the origin but no constant term, that is if
\begin{equation}
	a(\tau) = \sum_{j=1}^J a_j \frac{1}{(\tau-i0)^j} \bmod \tau C^\infty(-\epsilon,+\epsilon)
\end{equation}
in some neighborhood $(-\epsilon,+\epsilon)$ of $\tau = 0$, then we can find a polynomial $Z(\sigma) \in \bbC[\sigma]$ whose Fourier transform is $a_J(\tau-i0)^{-J}+\cdots+a_1(\tau-i0)^{-1}$. For this, we note that 
\begin{equation}
	\calF : \{\Theta(\sigma) Z(\sigma)\in \bbC[\sigma]:\operatorname{deg} Z \leq J-1\} \to \{ a_J(\tau-i0)^{-J}+\cdots+a_1(\tau-i0)^{-1}:a_J,\ldots,a_1\in \bbC\}
\end{equation}
is bijective.  Since we would like to be explicit in stating asymptotics, we record the exact correspondence:
\begin{proposition}
	\label{prop:pole_calc} 
	Given $J\in \bbN^+$ and $a_J,\ldots,a_1 \in \bbC$, the unique $Z(\sigma)\in \bbC[\sigma]$ such that 
	\begin{equation}
		\calF(\Theta Z)(\tau) = \sum_{j=1}^J a_j \frac{1}{(\tau-i0)^j} 
	\end{equation}
	is given by $Z(\sigma) = \sum_{j=0}^{J-1} i^{j+1} a_{j+1} \sigma^j/j!$.
\end{proposition}
\begin{proof}
	We see that $\calF (\Theta Z) = Z(D)  \calF \Theta $, where $D=  i \partial$. (So $Z(D)$ is a differential operator.) We recall that $\calF \Theta = -i (\tau-i0)^{-1}$. It follows that $Z$ has degree $J-1$, so $Z(\sigma)=b_{J-1} \sigma^{J-1}+\cdots+b_0$ for some $b_0,\ldots,b_{J-1} \in \bbC$. 
	\begin{equation}
		Z(D) \calF \Theta = \sum_{j=1}^J i^{j}(-1)^{j}(j-1)!\cdot b_{j-1} \frac{1}{(\tau -i0)^j}.
	\end{equation}
	We conclude that $b_j =  i^{j+1}  a_{j+1} / j!$ for each $j=0,\ldots,J-1$. 
\end{proof}
The analogous formula holds for nonintegral $j$, in which case we interpret $i^j$ as $e^{\pi  i j/2}$ and $j!$ as $\Gamma(j+1)$ --- cf.\ \cite[Vol.\ 1, pg. 359]{GS}. (Note that in \cite{GS}, the authors use the reverse of our conventions for $\calF,\calF^{-1}$.) 
\subsection{Conclusion} 
\begin{theorem}
	\label{thm:Ingham} 
	Suppose that $\alpha \in \scrS'(\bbR^{\geq 0}) \cap L^1_{\mathrm{loc}}(\bbR^{\geq 0})$, and let $a= \calF \alpha$, $\chi_0 \in C_{\mathrm{c}}^\infty(\bbR^+)$ with $\int_0^\infty \chi_0(\sigma) \dd \sigma = 1$. 
	
	Suppose further that we are given some distributions $a_0,b,c,g,p \in \scrS'(\bbR)$ of the form 
	\begin{enumerate}
		\item $\smash{a_0(\tau) = \sum_{j \in \calJ} a_{0}^{(j)} (\tau-i0)^{-j}}$ for some finite subset $\calJ\subset \bbR^{+}$ of positive real numbers and some indexed collection $\smash{\{a_0^{(j)}\}_{j\in \calJ}\subset \bbC}$,
		\item $b = \calF \beta$ for $\beta \in  \scrS'(\bbR^{\geq 0})\cap L^1_{\mathrm{loc}}(\bbR^{\geq 0})$, 
		\item $c= \calF\gamma$ for $\gamma \in  \scrS'(\bbR^{\geq 0})\cap L^1_{\mathrm{loc}}(\bbR^{\geq 0})$, such that $c$ is continuous in some neighborhood of the origin, 
		\item $g\in  \tau L^{1,N}_{\mathrm{loc}}(\bbR_\tau)$ with discrete singular support, for some $N\in \bbN$, 
		\item $p \in \scrS'(\bbR)$  with discrete singular support not containing the origin, vanishing at the origin, 
	\end{enumerate} 
	such that $a(\tau)=a_{00} \calF \chi_0 + a_0(\tau)+b(\tau)+c(\tau)$ for some $a_{00} \in \bbC$, where $c=g+p$ and $\gamma + \calF^{-1} p \in \langle \sigma \rangle^K L^\infty(\bbR_\sigma)$ for some $K\geq 0$ (Tauberian hypothesis). Suppose moreover that we have a nondecreasing continuous function $\frakR:\bbR^+\to \bbR^+$ such that 
	\begin{itemize}
		\item $g(\tau)$ satisfies $\lVert \tau^{-1} g(\tau) \rVert_{L^{1,N}[-T,+T]} \leq \frakR(T)$ for all $T\geq 1$ and
		\item $p$ satisfies 
		\begin{equation}
			\sup_{\Sigma>1,T>1} \frac{\Sigma^N }{ \frakR(T)} \Big|\int_{-T}^{+T} \frac{p(\tau)}{\tau}  e^{i \tau \Sigma} \Big(1 - \frac{\tau^2}{T^2} \Big)^M  \dd \tau \Big| < \infty 
			\label{eq:exa}
		\end{equation}
	for some $M \geq \operatorname{max}\{K+1,N\}$. 
	\end{itemize} 
	Then, for any $\Lambda>0$, 
	\begin{equation}
		\int_0^\Sigma \alpha(\sigma) \dd \sigma  = a_{00}+ \sum_{j\in \calJ, j\neq 0} \frac{e^{i\pi j/2}}{j!} a_0^{(j)} \Sigma^{j} + \int_0^\Sigma \beta(\sigma) \dd \sigma  + O(\Sigma^{K}  \calR(\Sigma)^{-1})
		\label{eq:g34}
	\end{equation}
	as $\Sigma\to\infty$, where $\calR(\Sigma)= \max\{1,\sup\{T>0 : T  \frakR(T) \leq \Lambda \Sigma^{N-K} \}\}$. 
\end{theorem}
\begin{proof}
	We apply \Cref{prop:setup} with $a_{00} \calF\chi_0 +a_0(\tau)+b(\tau)$ in place of $b(\tau)$. Then, for any $T\notin \operatorname{singsupp} c$,
	\begin{multline}
		\int_0^\Sigma \alpha(\sigma) \dd \sigma = \int_0^\Sigma \Big(  a_{00} \chi_0(\sigma) +\calF^{-1} a_0 (\sigma) +\beta(\sigma)\Big) \dd \sigma \\ + I_1(T,\Sigma,M) + I_2(T,\Sigma,M)+I_3(T,\Sigma,M) 
	\end{multline}
	where $I_1=I_1(T,\Sigma,M),I_2=I_2(T,\Sigma,M),I_3=I_3(T,\Sigma,M)$ are as in \Cref{prop:setup}. Since $\int_0^\infty \chi_0(\sigma) \dd \sigma = 1$, the first term above contributes an $a_{00}$ to the sum for sufficiently large $\Sigma$. By \Cref{prop:pole_calc},  
	\begin{equation}
		\calF^{-1} a_0 (\sigma) = \Theta(\sigma) \sum_{j\in \calJ} i^{j}\frac{a_{0}^{(j)}}{(j-1)!} \sigma^{j-1}. 
	\end{equation}
	So, $\int_0^\Sigma \calF^{-1} a_0(\sigma) = \sum_{j\in \calJ} i^{j} a_{0}^{(j)}  \Sigma^{j} / j!$. So, for sufficiently large $\Sigma$ and $T\geq 1$, 
	\begin{align} 
	\int_0^\Sigma \alpha(\sigma) \dd \sigma  &= a_{00}+ \sum_{j\in \calJ, j\neq 0} \frac{i^{j}}{j!} a_0^{(j)} \Sigma^{j} + \int_0^\Sigma \beta(\sigma) \dd \sigma + I_1+I_2+I_3 \\
	&= a_{00}+ \sum_{j\in \calJ, j\neq 0} \frac{i^{j}}{j!} a_0^{(j)} \Sigma^{j} + \int_0^\Sigma \beta(\sigma) \dd \sigma + O(\Sigma^K T^{-1})+I_3 \\
	&= a_{00}+ \sum_{j\in \calJ, j\neq 0} \frac{i^{j}}{j!} a_0^{(j)} \Sigma^{j} + \int_0^\Sigma \beta(\sigma) \dd \sigma + O(\Sigma^K T^{-1})+O(\Sigma^{-N} \frakR(T)) 
	\label{eq:l59}
	\end{align} 
	using \Cref{prop:setup_base}, \Cref{prop:setup2}. The constants in the big-Os do not depend on $\Sigma,T$. 
	Here we decomposed 
	\begin{equation}
		I_3(T,\Sigma,M) = \frac{1}{2\pi i} \int_{-T}^{+T} \frac{g(\tau)}{\tau}  e^{i \tau \Sigma} \Big(1 - \frac{\tau^2}{T^2} \Big)^M  \dd \tau +  \frac{1}{2\pi i} \int_{-T}^{+T} \frac{p(\tau)}{\tau}  e^{i \tau \Sigma} \Big(1 - \frac{\tau^2}{T^2} \Big)^M  \dd \tau,
	\end{equation}
	bounded the first using \Cref{prop:setup2}, and bounded the second using the assumption \cref{eq:exa}. 
	If $T=\calR(\Sigma)$, then (say, for $\Sigma\geq 1$)
	\begin{equation}
		O(\Sigma^K T^{-1}) = O(\Sigma^K \calR^{-1}(\Sigma)) \geq  O(\Sigma^{-N} \frakR(T)) . 
		\label{eq:misc_ooo}
	\end{equation} 
	While $\calR(\Sigma)$ might be in $\operatorname{singsupp} c$ (in which case we cannot take $T=\calR(\Sigma)$ and appeal to the results above), because $\frakR$ is continuous  and because the singular support of $c$ is discrete it is the case that for any sufficiently large $\Sigma> 0$ we can find $T\geq 1$ not in the singular support of $c$ such that \cref{eq:misc_ooo} holds in the sense that 
	\begin{equation} 
	I_3(T,\Sigma,M) \leq C \Sigma^K \calR^{-1}(\Sigma)
	\end{equation}  
	for some constant $C>0$ independent of $\Sigma,T$. 
	Then, \cref{eq:l59} reads 
	\begin{equation}
		\int_0^\Sigma \alpha(\sigma) \dd \sigma = a_{00}+ \sum_{j\in \calJ, j\neq 0} \frac{i^j}{j!} a_0^{(j)} \Sigma^{j} + \int_0^\Sigma \beta(\sigma) \dd \sigma  + O(\Sigma^{K} \calR^{-1}(\Sigma)), 
	\end{equation}
	which is \cref{eq:g34}.
\end{proof}

\section{Tauberian argument: Mean-to-Max}
\label{sec:second}

We now move on to the second step (the ``mean-to-max'' step) of Newman's argument, which involves attempting to extract information about $\alpha\in \cup_{K\geq 0} \langle \sigma \rangle^K L^\infty(\bbR^+_\sigma)$ from the asymptotics of 
\begin{equation} 
	A_\Sigma = \int_0^\Sigma \alpha(\sigma) \dd \sigma 
	\label{eq:As2}
\end{equation} 
as $\Sigma\to \infty$. In order to extract any asymptotics whatsoever, we need to know that $\alpha(\sigma)$ does not fluctuate too rapidly with $\sigma$. Rather than axiomatize this condition (which can be done \cite{KTextbook}), we specialize to the sort of functions considered in eigenvalue counting problems: we assume that $\alpha(\sigma)$ is of the form 
\begin{equation}
	\alpha(\sigma) = \frac{1}{\langle \sigma \rangle^{J}} (N(\sigma) - Z(\sigma) + E(\sigma)) 
	\label{eq:as2}
\end{equation}
for some nondecreasing piecewise-continuous $N :\bbR^+\to \bbR^+$, polynomial $Z(\sigma)\in \bbR[\sigma]$ of degree $L>0$ (which will in Weyl's law be $d$), and  $E\in C_{\mathrm{c}}^\infty(\bbR^+)$. Here $J\geq 0$ is a to-be-decided parameter, but -- given the range of admissible $K$ in \Cref{prop:setup_base} -- we will want to take $J\leq K_0$, where $K_0$ is such that $N(\sigma)-Z(\sigma) \in \langle \sigma \rangle^{K_0} L^\infty(\bbR_\sigma)$ is known to hold. (We write `$K_0$' instead of `$K$' in the previous sentence to distinguish the growth rate of $N(\sigma)-Z(\sigma)$ from that of $\alpha(\sigma)$, the latter being denoted $K$ as in the previous section.) (Increasing $J$ above $K_0$ yields weaker results.) In the application to Weyl's law, we take $J=K_0=d-1$. 
In this section, we will mainly take $J$ to be sufficiently large such that \Cref{thm:Ingham} will imply that $\smash{\lim_{\Sigma\to\infty} A_\Sigma = \lim_{\Sigma \to \infty} \int_0^\Sigma \alpha(\sigma) \dd \sigma}$ exists. Similar results to those proven under this assumption will hold as long as $A_\Sigma$ does not diverge or fluctuate too rapidly as $\Sigma\to\infty$. 

\begin{proposition}
	\label{prop:step_two}
	Given the setup above, $J\leq L$, suppose that $\lim_{\Sigma\to\infty} A_\Sigma$ exists and satisfies 
	\begin{align}
		A_\Sigma &= o\Big( \frac{1}{\calR(\Sigma)}\Big) + \lim_{\Sigma\to\infty} A_\Sigma 
		\label{eq:little_o}
		\intertext{or}
		A_\Sigma &= O\Big( \frac{1}{\calR(\Sigma)}\Big) + \lim_{\Sigma\to\infty} A_\Sigma 
		\label{eq:big_o}
	\end{align}
	for some nondecreasing $\calR: [0,\infty)\to (0,\infty)$ with $\lim_{\Sigma\to\infty} \calR(\Sigma)=\infty$. Set $\kappa = L-\Delta = (L+J-1)/2$, where $\Delta\geq 0$ is defined by $\Delta = (L-J+1)/2$. 
	
	Then, for sufficiently large $\sigma\geq 0$, 
	$N(\Sigma) = Z(\Sigma)+o(\Sigma^{\kappa} \calR(\Sigma/2)^{-1/2})$ if \cref{eq:little_o} holds and $N(\Sigma) = Z(\Sigma)+O(\Sigma^{\kappa} \calR(\Sigma/2)^{-1/2})$ if \cref{eq:big_o} holds.
\end{proposition}

\begin{proof}
	It evidently suffices to consider the case $E=0$. We use the inequalities 
	\begin{align}
		(N(\sigma) - Z(\sigma)) \min\Big\{ \frac{1}{\langle \sigma \rangle^{J-1}} , \frac{|N(\sigma) - Z(\sigma)|}{\langle \sigma \rangle^{J+L-1}} \Big\} &= O\Big( \sup_{\Sigma\geq \sigma} | A_\Sigma - A_\sigma |\Big)
		\label{eq:misc_921} \\
		 (Z(\sigma) - N(\sigma))\min\Big\{ \frac{1}{\langle \sigma \rangle^{J-1}}, \frac{|N(\sigma) - Z(\sigma)|}{\langle \sigma \rangle^{J+L-1}} \Big\} &= O\Big(\sup_{\sigma/2\leq \Sigma\leq \sigma} | A_\Sigma - A_\sigma | \Big),
		\label{eq:misc_123}
	\end{align}
	which hold for $\sigma>\sigma_0$, $\sigma_0$ sufficiently large (where the constant depends on $Z,\sigma_0$, but not on $\Sigma,\sigma$). 
	These bounds follow from the assumed monotonicity of $N$ and the slow variation of $Z$ (which is a consequence of the latter being a polynomial):
	\begin{itemize}
		\item for sufficiently large $\sigma_0$, $\sigma>\sigma_0$, if $N(\sigma) > Z(\sigma)$, 
		\begin{align}
			\begin{split} 
				A_\Sigma - A_\sigma  &= \int_\sigma^\Sigma \frac{1}{\langle \varsigma \rangle^J } (N(\varsigma) - Z(\varsigma)) \dd \varsigma \\
				&= \Big[ (N(\sigma) - Z(\sigma) ) \int_\sigma^\Sigma  \frac{1}{\langle \varsigma\rangle^J} \dd \varsigma + \int_\sigma^\Sigma \frac{1}{\langle \varsigma\rangle^J} (Z(\sigma)-Z(\varsigma)) \dd \varsigma   \Big],
			\end{split} 
		\end{align} 
		as long as $\sigma\leq \Sigma \leq 2\sigma$. So, 
		\begin{equation} 
				(N(\sigma) - Z(\sigma) ) \frac{ (\Sigma-\sigma)}{\langle \Sigma \rangle^J}  \dd \varsigma  = 
				O\Big( |A_\Sigma - A_\sigma| +\frac{1}{\langle \sigma\rangle^J}  \langle \sigma \rangle^{L-1} (\Sigma-\sigma)^2 \Big).  
				\label{eq:misc_k62}
		\end{equation} 
		Consequently, there exists a $C>0$ (dependent on $\sigma_0, Z$, but not on $\sigma$) such that for any $c>0$, if $c< (\Sigma-\sigma) \langle \sigma \rangle^{L-1} (N(\sigma)-Z(\sigma))^{-1}  <C$ and $\sigma\leq \Sigma\leq 2\sigma$, then we can absorb the second error on the right-hand side of \cref{eq:misc_k62} into the left-hand side to get 
		\begin{equation}
			A_\Sigma - A_\sigma \geq \Omega \Big( \frac{1}{\langle \sigma \rangle^{J+L-1}}  (N(\sigma)- Z(\sigma))^2  \Big)
			\label{eq:misc_812}
		\end{equation}
		(where the constant depends on $c,C$, along with $Z$). 
		On the other hand, if $|N(\sigma)-Z(\sigma)|\geq \langle \sigma \rangle^L$, then we can find $c',C'>0$ with $c'<C'<1$ such that if  $(1+c')\sigma\leq \Sigma\leq (1+C')\sigma$ then 
		\begin{equation}
			A_\Sigma - A_\sigma \geq \Omega \Big( \frac{1}{\langle \sigma \rangle^{J-1}}  (N(\sigma)- Z(\sigma))  \Big)
			\label{eq:misc_813}
		\end{equation}
		(where the constant depends on $c',C'$). 
		Taking $c\in (0,1/2)$, if $\sigma$ is sufficiently large and if $|N(\sigma)-Z(\sigma)|\leq \langle \sigma \rangle^L$ then there exists a $\Sigma$ such that $\sigma\leq \Sigma \leq 2\sigma$ and $c< (\Sigma-\sigma) \langle \sigma \rangle^{L-1} (N(\sigma)-Z(\sigma))^{-1}  <C$, in which case we conclude \cref{eq:misc_812}. Otherwise, we have \cref{eq:misc_813} for some $\Sigma \in [\sigma,2\sigma]$. In either case, \cref{eq:misc_921} holds (if $N(\sigma)>Z(\sigma)$). (And \cref{eq:misc_921} holds trivially if $N(\sigma) \leq Z(\sigma)$.) 
		\item 
		Similarly, if $N(\sigma)<Z(\sigma)$ for some sufficiently large $\sigma$, 
		\begin{align}
			A_\sigma - A_\Sigma &= \int_\Sigma^\sigma \frac{1}{\langle \varsigma\rangle^J} (N(\varsigma)-Z(\varsigma)) \dd \varsigma \\ 
			&=  \Big[ (N(\sigma) - Z(\sigma) ) \int_\Sigma^\sigma  \frac{1}{\langle \varsigma\rangle^J} \dd \varsigma + \int_\Sigma^\sigma \frac{1}{\langle \varsigma\rangle^J} (Z(\sigma)-Z(\varsigma)) \dd \varsigma   \Big] \\
			&\geq \Big[  (N(\sigma) - Z(\sigma) ) \frac{ (\sigma-\Sigma)}{\langle \Sigma \rangle^J} + O\Big( \frac{1}{\langle \sigma\rangle^J}  \langle \sigma \rangle^{L-1} (\sigma-\Sigma)^2 \Big) \dd \varsigma \Big]
		\end{align}
		as long as $\Sigma \in (\sigma/2,\sigma)$. We conclude as before that, for sufficiently large $\sigma$, \cref{eq:misc_123} holds.  
	\end{itemize}

	Combining \cref{eq:misc_921} and \cref{eq:misc_123}, and using the fact that $\calR$ is nondecreasing, we get the inequality 
	\begin{align}
		|N(\sigma) - Z(\sigma) | &\leq  O(\max\{ \langle \sigma \rangle^{J-1} \calR(\sigma/2)^{-1},\langle \sigma \rangle^{(J+L-1)/2} \calR(\sigma/2)^{-1/2}\}) \\ 
		 &\leq  O(\langle \sigma \rangle^{(J+L-1)/2} \calR(\sigma/2)^{-1/2}) 
		\intertext{or the inequality} 
		|N(\sigma) - Z(\sigma) | &\leq o(\max\{ \langle \sigma \rangle^{J-1}\calR(\sigma/2)^{-1}, \langle \sigma \rangle^{(J+L-1)/2} \calR(\sigma/2)^{-1/2}\}) \\
		&\leq o(\langle \sigma \rangle^{(J+L-1)/2} \calR(\sigma/2)^{-1/2}),
	\end{align}
	depending on which of  \cref{eq:little_o}, \cref{eq:big_o} hold. 
\end{proof}

We now prove \Cref{thm:main_1}, which we rewrite here as a proposition: 
\begin{proposition}
	\label{prop:main_1}
	Suppose that $N: \bbR^{\geq 0} \to (0,\infty)$ is a piecewise-continuous, nondecreasing function, $Z(\sigma)\in \bbC[\sigma]$ is a polynomial in $\sigma$ of degree at most $d\in \bbN$, and $E\in C_{\mathrm{c}}^\infty((0,\infty))$. 
	Let 
	\begin{equation} 
		\alpha(\sigma) = \begin{cases}
			\langle \sigma\rangle^{1-d} (N(\sigma)-Z(\sigma)+E(\sigma)) & (\sigma\geq 0) \\
			0 & (\sigma< 0),
		\end{cases}
	\end{equation}
	and suppose that the singular support of $\calF N$ and hence $a= \calF\alpha$ is discrete. Suppose further that 
	\begin{enumerate}
		\item $\alpha \in  L^\infty(\bbR)$, 
		\label{it:taub1}
		\item $a(\tau)$ is in $\tau L^{1,\ell}_{\mathrm{loc}}(\bbR_\tau)$ for some integer $\ell\geq 1$. 
		\label{it:ab1}
	\end{enumerate}
	Then, for any $\Lambda>0$, setting $\calR_{0,\Lambda}(\Sigma) = \max\{1, \sup\{T>0: T \lVert t^{-1} a(t) \rVert_{L^{1,\ell}[-T,+T]} \leq \Lambda \Sigma^\ell \}\}$, 
	\begin{equation} 
		N(\Sigma) = Z(\Sigma)  + O( \Sigma^{d-1} \calR_{0,\Lambda}(\Sigma)^{-1/2})
	\end{equation}  
	as $\Sigma\to\infty$. 
\end{proposition}
\begin{proof}
	We apply \Cref{thm:Ingham} with $a_{00},a_0,b,p = 0$, i.e.\ with $a = c=g$. The assumptions \cref{it:taub1}, \cref{it:ab1} give the hypotheses of that theorem. (Note that $a\in \tau L^{1,1}_{\mathrm{loc}}(\bbR_\tau)$ implies $a\in \tau C^0(\bbR_\tau)$, so in particular $a$ is continuous in some neighborhood of the origin, verifying the remaining hypothesis of \Cref{thm:Ingham}.) Then, per the conclusion of the theorem,
	\begin{align}
		A_\Sigma=\int_0^\Sigma \frac{1}{\langle \sigma \rangle^{d-1}}(N(\sigma) - Z(\sigma)) \dd \sigma &= \int_0^\Sigma \alpha(\sigma)\dd \sigma - \int_0^\Sigma \frac{E(\sigma)}{\langle \sigma \rangle^{d-1}}  \dd \sigma  \\ &=  -\int_0^\infty \frac{E(\sigma)}{\langle \sigma \rangle^{d-1}}  \dd \sigma  + O( \calR_{0,2^{1/\ell}\Lambda}(\Sigma)^{-1}) \\
		&=  -\int_0^\infty \frac{E(\sigma)}{\langle \sigma \rangle^{d-1}}  \dd \sigma  + O( \calR_{0,\Lambda}(2\Sigma)^{-1}). \label{eq:misc_765}
	\end{align}
	Since $\lim_{\Sigma\to\infty} \calR_{0,\Lambda}(\Sigma)=\infty$, \cref{eq:misc_765} implies that $\lim_{\Sigma\to\infty} A_\Sigma = - \int_0^\infty E(\sigma) \langle \sigma \rangle^{1-d} \dd \sigma $ exists.
	Applying \Cref{prop:step_two}, We conclude that $N(\Sigma) = Z(\Sigma) + O(\Sigma^{d-1} \calR_{0,\Lambda}(\Sigma)^{-1/2} )$.
\end{proof}

\section{Application to Weyl's law}
\label{sec:final}

Let $(M,g)$ denote a compact Riemannian manifold of dimension $d\geq 1$, and let $0=\lambda_0 \leq \lambda_1 \leq \lambda_2\leq \cdots$ denote the eigenvalues of the positive semidefinite Laplace-Beltrami operator (counted with multiplicity), and for convenience we assume $\operatorname{Vol}_g M = 1$.
Consider the following well-defined regular Borel measure on the line $\bbR_\sigma$:
\begin{equation}
	\mu_{1/2}(\sigma) = \sum_{n=0}^\infty \delta(\sigma-\lambda_n^{1/2}) = \sum_{n=0}^\infty \delta(\sigma-\sigma_n), 
\end{equation}
where $\sigma_n = \lambda_n^{1/2}$. 
Note that $\mu_{1/2}$, considered as a map $\smash{\scrS(\bbR)\ni \varphi \mapsto \int_{-\infty}^{+\infty} \varphi(\sigma) \dd \mu_{1/2}(\sigma)}$, is in $\scrS'(\bbR)$ (e.g.\ as a consequence of the weakest form of Weyl's law). This measure is related to the eigenvalue counting function $N(\lambda) = \{n\in \bbN: \lambda_n \leq \lambda \}$ by the distributional identity $\partial_\sigma N_{1/2}(\sigma) = \mu_{1/2}(\sigma)$, where 
\begin{equation}
	N_{1/2}(\sigma)= \#\{ n : \lambda_n^{1/2} = \sigma_n\leq \sigma \} = N(\sigma^2). 
\end{equation}
Now suppose that the nontrivial geodesic loops in $M$ are at most $\delta$-fold degenerate. (`Geodesic loop' means smooth geodesic loop, so a periodic trajectory of the Hamiltonian flow in phase space.) Here 
\begin{equation} 
	\delta\in \{0,\ldots,2d-1\} 
\end{equation} 
is roughly the Hausdorff dimension of the set of points in the sphere bundle $\bbS M$ lying along a (nontrivial) geodesic loop. (This dimension will always be positive, unless $M$ just so happens to have no nontrivial geodesic loops.)
Suppose that $(M,g)$ is a nonpathological (see the precise hypotheses of \cite[Theorem 4.5]{DG}), smooth, compact Riemannian manifold without boundary of dimension $d\geq 1$, whose nontrivial geodesic loops are at most $\delta$-fold degenerate. Then, as shown in \cite{Chazarain1974}\cite{Hormander68}\cite{DG}, letting $\calT \subseteq \bbR$ be the set of lengths of geodesic loops (including loops consisting of multiple laps around a single loop):
\begin{enumerate}
	\item $\calT$ is discrete and has no accumulation points.
	\label{DG:1}
	\item The singular support of the half-wave trace $\operatorname{HWT}_{M,g}=\calF\mu_{1/2}$ is $\calT$. 
	\label{DG:2}
	\item $0 \in \calT$, and given an open neighborhood $U \subset \bbR$ of $0$ whose closure is disjoint from $\calT\backslash \{0\}$, 
	\[
	\calF\mu_{1/2}|_U(\tau) = \sum_{j=0}^{d-1}  \frac{Z_j}{(\tau-i0)^{d-j}}+ E(\tau)
	\]
	for some $Z_0,\ldots,Z_{d-1} \in \bbC$, $Z_{0}\neq 0$, and $E \in C^\infty(U)$. Moreover, $Z_1,Z_3,Z_5,\cdots = 0$ \cite[Proposition 2.1]{DG}.
	\label{DG:3}
	\item Given any $T \in \calT\backslash \{0\}$ and an open neighborhood $V \subseteq \bbR$ of $T$ whose closure is disjoint from $\calT\backslash \{T\}$, 
	\[
	\calF\mu_{1/2}|_V(\tau) = \sum_{j=0}^{\lfloor 2^{-1}(\delta-1) \rfloor} \frac{C_j}{(\tau-T-i0)^{-j+(\delta+1)/2}} + E(\tau) 
	\]
	for some $T$-dependent $C_0=C_0(T),\ldots,C_{\lfloor 2^{-1}(\delta-1) \rfloor} = C_{\lfloor 2^{-1}(\delta-1) \rfloor}(T) \in \bbC$, and function $E \in C^\infty(V)$.  
	\label{DG:4}
\end{enumerate}
\begin{figure}
	\begin{center}
		\includegraphics[width=.5\linewidth]{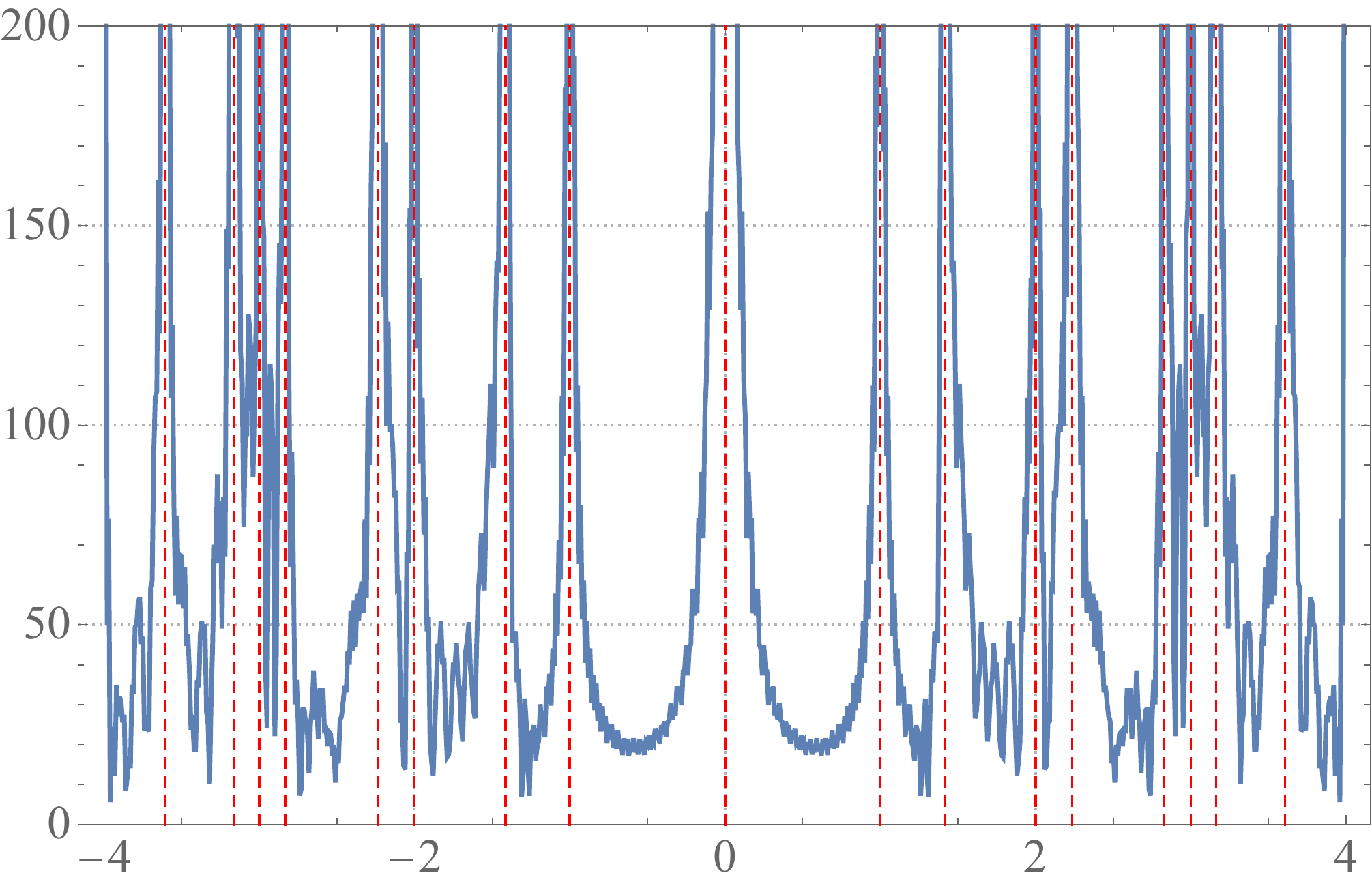}
	\end{center}
\caption{A numerical approximation of $|\operatorname{HWT}_{M,g}(\tau)|$ on the flat 2-torus $\bbT^2$, scaled so that $\smash{\calT=\{\pm (n^2+m^2)^{1/2}:n,m\in \bbN\}}$. Here, we are applying a cut-off in frequency space, summing up the contributions coming from the first 10201 eigenvalues. Dashed red lines mark the elements of $\calT \cap (-4,+4)$. }
\end{figure}

Referring to the statement of \cite[Theorem 4.5]{DG}, $\delta = \max\{d_j(T):j=1,\ldots,r, T\in \calT\backslash \{0\}\}$. 
In fact, H\"ormander and Duistermaat \& Guillemin provide explicit formulas for the first couple of these coefficients. 
The very precise information on the singularities of $\calF\mu_{1/2}$ given by the Duistermaat-Guillemin result can serve as the local input to Newman's Tauberian argument.
For each $s \in \bbR $, $\varepsilon>0$, and open $U\subseteq \bbR$, let $\langle D \rangle^{-s} L^{1+\varepsilon}_{\mathrm{loc}}(U) \subseteq \scrD'(\bbR)$ denote the subspace of tempered distributions $u$ such that whenever $\chi \in C_{\mathrm{c}}^\infty(\bbR)$ is supported in $U$, $\langle D \rangle^s \chi u \in L^{1+\varepsilon}(\bbR)$. (This notation is slightly abusive, since according to standard usage a statement such as ``$u\in L^{1+\varepsilon}_{\mathrm{loc}}(\bbR)$'' is not supposed to involve any temperedness of $u$ --- but since $\langle D \rangle^s$ is only defined on tempered distributions (except for $s\in 2 \bbN$), this should not cause any confusion below.)

\begin{lemma}
	\label{lem:Sobolev_agreement}
	For any $\varepsilon>0$ and $s\in \bbR$, if $\ell\in \bbN$ satisfies $\ell\leq s$, if $u\in \langle D \rangle^{-s} L^{1+\varepsilon}_{\mathrm{loc}}(U)$ then $\chi u \in L^{1+\varepsilon,\ell}(\bbR)$ for all $\chi \in C_{\mathrm{c}}^\infty(U)$. 
\end{lemma}
\begin{proof}
	We have that $\partial^j \chi u = \partial^j \langle D \rangle^{-s} \langle D \rangle^{s} \chi u \in \partial^j \langle D \rangle^{-s} L^{1+\varepsilon}(\bbR)$. For $j\leq \ell$, $ \partial^j \langle D \rangle^{-s}$ is a Kohn-Nirenberg $\Psi$DO of nonpositive order, which implies that it maps compactly supported elements of $L^{1+\varepsilon}(\bbR)$ to $L^{1+\varepsilon}(\bbR)$ (using the $L^p$-boundedness of Kohn-Nirenberg $\Psi$DOs for $p>1$; see e.g.\ \cite[Theorem 6.5.5]{SimomHarmonic}). (Of course, one does not need all this theory to deal with the special case of $\langle D \rangle^{-s}$. See \cite[\S6.3]{SimomHarmonic} for a direct treatment.) So, $\partial^j \chi u \in L^{1+\varepsilon}(\bbR)$ for each $j=1,\ldots,\ell$, and we can therefore conclude that $\chi u \in L^{1+\varepsilon,\ell}(\bbR)$.  
\end{proof}

\begin{proposition}
	\label{prop:apriori}
	Suppose that $(M,g)$ is a nonpathological $d$-dimensional compact Riemannian manifold (in the sense of satisfying the hypotheses of \cite[Theorem 4.5]{DG}) whose nontrivial geodesic loops are at most $\delta$-fold degenerate, $\delta \in \{0,\ldots,2d-2\}$.

	Then,
	\begin{equation}
	\langle D_\tau \rangle^{1-d} (\tau^{-1} \operatorname{HWT}_{M,g}(\tau)) \in L^{1,\ell}_{\mathrm{loc}}(\bbR\backslash \{0\})
	\label{eq:nice_Sobolev_goal}
	\end{equation} 
	for $\ell = \lceil d - \delta/2-1/2 \rceil - 1$.
\end{proposition}

\begin{proof} 
	We first show that there exists a function $\epsilon:\bbR^+\to \bbR^+$ with $\lim_{\varepsilon\to 0^+} \epsilon(\varepsilon) = 0$ such that  the regularized half-wave trace $\langle D_\tau \rangle^{1-d} (\tau^{-1} \operatorname{HWT}_{M,g}(\tau))$ satisfies 
	\begin{equation} 
	\langle D_\tau \rangle^{1-d} (\tau^{-1} \operatorname{HWT}_{M,g}(\tau)) \in 
	\langle D\rangle^{-(d-\delta/2-1/2-\epsilon(\varepsilon))} L^{1+\varepsilon}_{\mathrm{loc}}(\bbR \backslash \{0\})
	\label{eq:misc_173}
	\end{equation} 
	for every $\varepsilon >0$. Given this, \cref{eq:nice_Sobolev_goal} follows from \Cref{lem:Sobolev_agreement}. 
	Set $\alpha(\sigma) = \langle \sigma \rangle^{1-d} N_{1/2}(\sigma)$. Since $\operatorname{HWT}_{M,g}(\tau) = - i \tau \calF N_{1/2}$,  \cref{eq:misc_173} is equivalent to 
	\begin{equation}
	\calF\alpha \in 
	\langle D\rangle^{-(d-\delta/2-1/2-\epsilon(\varepsilon))} L^{1+\varepsilon}_{\mathrm{loc}}(\bbR \backslash \{0\}). 
	\end{equation}

	Since $\calT$ is discrete, we may choose a partition of unity $\{\chi_n\}_{n=0}^\infty \subset C_{\mathrm{c}}^\infty(\bbR;[0,1])$ of $\bbR$ such that each $\chi_n$ contains at most one point of $\calT$ in its support and (if it contains such a point) is identically equal to one in some neighborhood of it. We index our partition of unity such that $\chi_0$ is the only element of this partition of unity whose support contains the origin. 
	We can write, for each $n \in \bbN$, 
	\begin{equation}
		\chi_n \calF \alpha = \chi_n \langle D \rangle^{1-d} \calF N_{1/2} = \chi_n \langle D \rangle^{1-d} \chi_n \calF N_{1/2} + \chi_n \langle D \rangle^{1-d} (1-\chi_n) \calF N_{1/2}. 
		\label{eq:k3}
	\end{equation}
	Since $\calF \mu_{1/2}(\tau) = - i \tau \calF N_{1/2}$, by the pseudolocality of $\langle D \rangle^{1-d}$ the second term in (\ref{eq:k3}), $\chi_n \langle D \rangle^{1-d} (1-\chi_n) \calF N_{1/2}$, is smooth. (If $\operatorname{supp} \chi_n \cap \calT= \varnothing$, then from $\operatorname{singsupp}\langle D \rangle^{1-d} (1-\chi_n) \calF N_{1/2} \subset \calT$ we conclude $\smash{\operatorname{singsupp} \chi_n \langle D \rangle^{1-d} (1-\chi_n) \calF N_{1/2} = \varnothing}$. If $\operatorname{supp} \chi_n \cap \calT = \{T\}$,  then $\operatorname{singsupp} (1-\chi_n) \calF N_{1/2} \subset \calT\backslash \{T\}$, so a similar conclusion holds. Alternatively, via the symbol calculus, the essential support of the operator $\chi_n \langle D \rangle^{1-d} (1-\chi_n)$ is disjoint from $\calT$.)
	It remains therefore to show that 
	\begin{equation}
		\chi_n \langle D \rangle^{1-d} \chi_n \calF N_{1/2} \in \langle D\rangle^{-(d-\delta/2-1/2-\epsilon)}L^{1+\varepsilon}_{\mathrm{loc}}(\bbR\backslash \{0\})
		\label{eq:n43}
	\end{equation}
	for each $n\neq 0$. This will hold if $\langle D \rangle^{1-d} \chi_n \calF N_{1/2} \in \langle D\rangle^{-(d-\delta/2-1/2-\epsilon)} L^{1+\varepsilon}_{\mathrm{loc}}(\bbR \backslash \{0\})$ holds.

	If $\chi_n$ contains no points of $\calT$ in its support, then by Chazarain's theorem, $\chi_n \calF\alpha \in C_{\mathrm{c}}^\infty(\bbR)$, so by the pseudolocality of $\langle D \rangle^{1-d}$, (\ref{eq:n43}) holds. Otherwise -- letting $T$ denote the sole point in $\calT \cap \operatorname{supp}\chi_n$ -- by the Duistermaat-Guillemin result, we can write 
	\begin{align}
		\chi_n(\tau) \calF\mu_{1/2}(\tau) &= \chi_n(\tau)\sum_{j=0}^{\lfloor 2^{-1}(\delta-1) \rfloor}  C_j (\tau-T-i0)^{j0(\delta+1)/2}+ E(\tau) 
		\label{eq:kk}
		\intertext{for some $E \in C^\infty_{\mathrm{c}}(\bbR)$, depending on $n$. Since $\chi_n$ is identically equal to one in some neighborhood of $T$, \cref{eq:kk} actually implies that  } 
		\chi_n(\tau) \calF\mu_{1/2}(\tau)  &= \sum_{j=0}^{\lfloor 2^{-1}(\delta-1) \rfloor} C_j(\tau-T-i0)^{j-(\delta+1)/2}  + F(\tau) 
		\label{eq:11}
	\end{align}
	for some  $F \in C_{\mathrm{c}}^\infty(\bbR)$. 
	Applying the pseudodifferential operator $\langle D\rangle^{1-d}$ to both sides \cref{eq:11}, 
	\begin{equation} \langle D \rangle^{1-d} \chi_n \calF \mu_{1/2}(\tau) = \sum_{j=0}^{\lfloor 2^{-1}(\delta-1) \rfloor}C_j \langle D \rangle^{1-d}(\tau-T-i0)^{j-(\delta+1)/2}+ \langle D \rangle^{1-d} F(\tau).
	\end{equation} 
	By the mapping properties of $\Psi$DOs (or equivalently pseudolocality), $\smash{\langle D \rangle^{1-d} F(\tau) \in C^\infty(\bbR)}$. It therefore suffices to observe that 
	\begin{equation} 
	\langle D \rangle^{1-d} (\tau-T-i0)^{j-(\delta+1)/2} \in \langle D\rangle^{-(d-\delta/2-1/2-\epsilon)} L^{1+\varepsilon}_{\mathrm{loc}}(\bbR )
	\end{equation} 
	holds, for each $j=0,\ldots,\lfloor 2^{-1}(\delta-1) \rfloor$, which is equivalent to the inclusion $(\tau-T-i0)^{j-(\delta+1)/2} \in \langle D\rangle^{\delta/2-1/2+\epsilon} L^{1}_{\mathrm{loc}}(\bbR )$ holding for each $j$. 
	By the translation invariance of $\langle D \rangle$ and its fractional powers, it suffices to show that $\langle D \rangle^{1-d} (\tau-i0)^{j-(\delta+1)/2} \in \langle D\rangle^{-(d-\delta/2-1/2-\epsilon)} L^{1+\varepsilon}_{\mathrm{loc}}(\bbR )$ holds, for each $j=0,\ldots,\lfloor 2^{-1}(\delta-1)\rfloor$.  
	Via the Taylor series expansion of $\langle \sigma \rangle^{-\ell}$ around $\sigma=\infty$, we see that $ \langle D\rangle^{-(\delta/2-1/2+\epsilon(\varepsilon))} (\tau-i0)^{j-(\delta+1)/2}= (\tau-i0)^{-1+\epsilon(\varepsilon)/2+j} G(\tau)$ for some continuous $G(\tau)$. This is in $L^{1+\varepsilon}_{\mathrm{loc}}(\bbR)$ for each $j = 0,\ldots,\lfloor 2^{-1}(\delta-1) \rfloor$ if and only if it is for $j=0$. This holds if $(1-\epsilon/2)(1+\varepsilon)<1$, so we can find $\epsilon(\varepsilon)$ with the desired properties. 
\end{proof}

Finally, \Cref{thm:main}, restated here as a proposition:
\begin{proposition}
	Suppose that $(M,g)$ is a compact Riemannian manifold such that the regularized half-wave trace $\tau^{-1} \langle D_\tau \rangle^{1-d} (\tau^{-1} \operatorname{HWT}_{M,g}(\tau))$ is in $L_{\mathrm{loc}}^{1,\ell}(\bbR\backslash \{0\})$ for some $\ell\in \bbN^+$. 
	
	Then, for some polynomial $Z_0(\sigma) \in \bbC[\sigma]$ of degree at most $d-2$, 
	\begin{equation}
	N_{1/2}(\Sigma) = (2\pi)^{-d} \operatorname{Vol}_g(M) \operatorname{Vol}(\bbB^d) \Sigma^d + Z_0(\Sigma)  + O(\Sigma^{d-1} \calR(\Sigma)^{-1/2} )
	\label{eq:misc_987}
	\end{equation}
	as $\Sigma\to\infty$, where $\calR(\Sigma) = \max\{1,\sup\{T>0: T \lVert \tau^{-1} \langle D_\tau \rangle^{1-d} (\tau^{-1} \operatorname{HWT}_{M,g}(\tau)) \rVert_{L^{1,\ell}[1,T]_\tau} \leq \Sigma^\ell \}\}$. 
\end{proposition}
\begin{proof}
	We can choose a $Z\in \bbC[\sigma]$, $E\in C_{\mathrm{c}}^\infty(\bbR^+)$ such that, setting 
	\begin{equation} 
	\alpha(\sigma) = 
	\begin{cases}
		0 & (\sigma< 0) \\
		\langle \sigma \rangle^{1-d} (N_{1/2}(\sigma) - Z( \sigma) + E(\sigma) ) & (\sigma\geq 0),
	\end{cases}
	\end{equation} 
	$a=\calF \alpha$ is smooth in some neighborhood of the origin and vanishing to first order.  Moreover, the leading term of $Z$ is $(2\pi)^{-d}\operatorname{Vol}_g(M) \operatorname{Vol}(\bbB^d) \sigma^d$ (by the same computation as in the proof of Weyl's law). From the absence of half of the terms in \Cref{DG:3} (note the different usage of `$Z$') -- see also \cite[Equation 2.4]{DG} -- we see that the subleading part $Z_0$ of $Z$ has degree at most $d-2$. 
	
	We now apply \Cref{prop:main_1}, the result being 
	\begin{equation}
	N_{1/2}(\Sigma)=(2\pi)^{-d} \operatorname{Vol}_g(M) \operatorname{Vol}(\bbB^d) \Sigma^d + Z_0(\Sigma)  + O(\Sigma^{d-1} \calR_{0,\Lambda}(\Sigma)^{-1/2} )
	\label{eq:misc_999}
	\end{equation}
	for $\calR_{0,\Lambda}(\Sigma) = \sup\{T>1: T \lVert t^{-1} a(t) \rVert_{L^{1,\ell}[-T,+T]} \leq \Lambda \Sigma^\ell \}$ and $\Lambda>0$ arbitrary. Away from the origin 
	\begin{equation}
	a(\tau) = i\langle D _\tau \rangle^{1-d} ( \tau^{-1} \operatorname{HWT}_{M,g}(\tau))  + \langle D_\tau \rangle^{1-d}\calF E(\tau).
	\end{equation}
	As noted above, $\langle D_\tau \rangle^{1-d}\calF E(\tau)$ is Schwartz. Since the half-wave trace satisfies $\operatorname{HWT}_{M,g}(-\tau) = \operatorname{HWT}_{M,g}(\tau)^*$, we can write 
	\begin{equation}
	\lVert t^{-1} a(t) \rVert_{L^{1,\ell}[-T,+T]} = \lVert t^{-1} a(t) \rVert_{L^{1,\ell}[-1,+1]} + 2\lVert t^{-1} a(t) \rVert_{L^{1,\ell}[1,T]}
	\end{equation}
	for $T\geq 1$. Also, 
	\begin{multline}
	\lVert \tau^{-1} a(\tau) \rVert_{L^{1,\ell}[1,T]} \leq \lVert \tau^{-1} \langle D _\tau \rangle^{1-d} ( \tau^{-1} \operatorname{HWT}_{M,g}(\tau)) \rVert_{L^{1,\ell}[1,T]} \\  + \lVert \tau^{-1} \langle D_\tau \rangle^{1-d}\calF E(\tau) \rVert_{L^{1,\ell}[1,\infty)}.
	\end{multline}
	Note that $a$ cannot be identically zero on $[1,\infty]$, so for some $T_0>0$ there exists some $c>0$ such that $\lVert \tau^{-1} a(\tau) \rVert_{L^{1,\ell}[-1,+1]} ,\lVert \tau^{-1} \langle D_\tau \rangle^{1-d}\calF E(\tau) \rVert_{L^{1,\ell}[1,\infty)} \leq c \lVert \tau^{-1} \langle D _\tau \rangle^{1-d} ( \tau^{-1} \operatorname{HWT}_{M,g}(\tau)) \rVert_{L^{1,\ell}[1,T]}$ for $T\geq T_0$.

	For $\Lambda>2+3c$, 
	\begin{align}
	\calR_{0,\Lambda}(\Sigma)  &\geq \sup\{1,T\geq T_0: T \lVert \tau^{-1}  \langle D _\tau \rangle^{1-d} ( \tau^{-1} \operatorname{HWT}_{M,g}(\tau))  \rVert_{L^{1,\ell}[1,T]} \leq (2+3c)^{-1}\Lambda \Sigma^\ell \} \\
	&\geq \sup\{1,T>0: T \lVert \tau^{-1}  \langle D _\tau \rangle^{1-d} ( \tau^{-1} \operatorname{HWT}_{M,g}(\tau))  \rVert_{L^{1,\ell}[1,T]} \leq  \Sigma^\ell \} = \calR(\Sigma). 
	\end{align}
	Substituting this into \cref{eq:misc_999} yields \cref{eq:misc_987}. 
\end{proof}

With \Cref{prop:apriori}, this yields \Cref{thm:final}. 

\section*{Acknowledgements}

This work was partially supported by a Hertz fellowship. Thanks go to  Gregory Debruyne, Elliott Fairchild, Peter Hintz, and Richard Melrose for helpful comments or conversations.

\printbibliography
	
\end{document}